\documentclass{amsart}

\usepackage{amsmath,amssymb}
\usepackage{graphicx}
\usepackage{epstopdf}

\newcommand{\arr}{\longrightarrow}
\newcommand{\mapdown}[1]%
{\Big\downarrow\rlap{$\vcenter{\hbox{$\scriptstyle#1$}}$}}

\newcommand{\wt}{\widetilde}

\newcommand{\Gr}{\mathfrak{G}}

\newcommand{\M}{\mathcal{M}}
\newcommand{\be}{\mathsf{s}}
\newcommand{\en}{\mathsf{r}}
\newcommand{\X}{\mathcal{X}}

\newcommand{\R}{\mathbb{R}}
\newcommand{\Z}{\mathbb{Z}}

\newcommand{\dad}{\mathop{\mathrm{dad}}}

\newcommand{\diam}{\mathop{\mathrm{diam}}}

\newcommand{\alb}{\mathsf{X}}
\newcommand{\xs}{\alb^*}
\newcommand{\xo}{\alb^\omega}
\newcommand{\xmo}{\alb^{-\omega}}
\newcommand{\nuke}{\mathcal{N}}
\newcommand{\img}[1]{\mathop{\mathrm{IMG}\left(#1\right)}}
\newcommand{\til}{\mathcal{T}}
\newcommand{\lims}[1][G]{\mathcal{J}_{#1}}
\newcommand{\limg}[1][G]{\mathcal{X}_{#1}}
\newcommand{\bim}{\mathfrak{B}}

\newcommand{\symm}{\mathsf{Symm}}
\newcommand{\si}{\mathsf{s}}

\newtheorem{theorem}{Theorem}[section]
\newtheorem{lemma}[theorem]{Lemma}
\newtheorem{proposition}[theorem]{Proposition}

\theoremstyle{definition}
\newtheorem{defi}[theorem]{Definition}
\newtheorem{example}[theorem]{Example}

\title{Self-similar groups and topological dimension}
\author{Volodymyr Nekrashevych}

\begin{document}
\begin{abstract}
We give various characterizations of the covering dimension of the limit space of a contracting self-similar group. In particular, we show that it is equal to the minimal dimension of a contracting affine model, to the asymptotic dimension of the orbital graphs of its action on the boundary of the tree, to the dynamical asymptotic dimension of its groupoid of germs, and give a combinatorial description of the dimension in terms of a partition of the levels of the tree. We analyze separately the one-dimensional example. In particular, we show that the limit space is one-dimensional if and only if the group is the faithful quotient of a self-similar contracting virtually free group.
\end{abstract}

\maketitle
\tableofcontents
\section{Introduction}

A faithful action of a group $G$ on the set $\xs$ of finite words over a finite alphabet $\alb$ is called \emph{self-similar} if for every $g\in G$ and every $v\in\xs$ there exists $h\in G$ such that 
\[g(vw)=g(v)h(w)\]
for all $w\in\xs$. The element $h$ is called the \emph{section} of $g$ in $v$ and is denoted $g|_v$.

We say that the group is \emph{contracting} if asymptotically (i.e., for long $g$ and long $v$) the element $g|_v$ is shorter than $g$ (e.g., in terms of the length with respect to a finite generating set of $G$).

Every locally expanding covering map $f:\mathcal{J}\arr\mathcal{J}$ of a compact metric space is uniquely encoded by the associated \emph{iterated monodromy group} $\img{f}$, which is a contracting self-similar group in a natural way.

Conversely, one can associate with every contracting self-similar group its \emph{limit dynamical system}, which is the original expanding covering map if the group is the iterated monodromy group. In general, one has to consider expanding covering maps of \emph{orbispaces}: spaces locally represented as quotients by actions of finite groups. Then we get a functorial bijection between expanding covering maps and contracting self-similar groups, up to natural equivalences. For more on the theory of iterated monodromy groups and expanding covering maps, see~\cite{nek:book,nek:dyngroups}.

Expanding covering maps have rich geometric structures (see~\cite{haisinskypilgrim}, for example). One of natural invariants is the topological  dimension of the space, which is always finite in this case. 
The aim of this paper is to show how the topological dimension of the limit dynamical system can be interpreted in terms of the self-similar group.

By the results of~\cite{nek:models}, every expanding covering map $f:\mathcal{J}\arr\mathcal{J}$ can be approximated by iterations of a \emph{contracting correspondence} $\tilde f, \iota:\Gamma_1\arr\Gamma$, where $\Gamma$ is a simplicial complex (\emph{complex of groups} in general), $\tilde f$ is a finite degree covering map, and $\iota$ is a piecewise affine map, such that $\tilde f$ is a local isometry and $\iota$ is strictly contracting. We define then the complex $\Gamma_n$ as the space of $n$-tuples $(x_1, x_2, \ldots, x_n)$ such that $\tilde f(x_i)=\iota(x_{i+1})$. It is shown in~\cite{nek:models} that if a naturally defined \emph{iterated monodromy group} of $\tilde f, \iota:\Gamma_1\arr\Gamma$ is the same as the iterated monodromy group of $f$, then the inverse limit of the maps $\iota_n:\Gamma_{n+1}\arr\Gamma_n: (x_1, x_2, \ldots, x_{n+1})\mapsto (x_1, x_2 \ldots, x_n)$ is homeomorphic to $\mathcal{J}$. Moreover, the map $f:\mathcal{J}\arr\mathcal{J}$ is topologically conjugate to the map on the limit induced by the maps $\tilde f_n:(x_1, x_2, \ldots, x_{n+1})\mapsto (x_2, x_3, \ldots, x_{n+1})$. In fact, the map $f:\mathcal{J}\arr\mathcal{J}$ is topologically conjugate to the limit dynamical system of the iterated monodromy group of the correspondence.
We call such a correspondence $\tilde f, \iota:\Gamma_1\arr\Gamma$ a \emph{contracting model} of the covering map.

The first result of the paper is Theorem~\ref{th:contractingmodeldim} stating that the minimal dimension of a contracting model of an expanding covering map $f:\mathcal{J}\arr\mathcal{J}$ is equal to the topological dimension of $\mathcal{J}$. 

The second result (Theorem~\ref{th:dad}) is equality between the topological dimension of the limit space of a contracting group $G$ and the \emph{dynamical asymptotic dimension} of its action on the boundary $\xo$ of the tree (more precisely of the groupoid of germs of the action). The notion of dynamical asymptotic dimension (d.a.d.) was developed in~\cite{GWY:dad} as a tool for studying cross-product $C^*$-algebras and their classification using $K$-theory.

The definition of the dynamical asymptotic dimension of the groupoid of germs of the action of $G$ on $\xo$ can be easily reformulated in terms of existence of a partition of a level $\alb^n$ of the tree $\xs$ into subsets satisfying a finiteness condition for the action of $G$. This reformulation is also included in the formulation of Theorem~\ref{th:dad}. It gives a purely combinatorial method of finding upper bounds on the topological dimension of the limit space.

The relation between the groupoid of germs of the action of $G$ on $\xo$ and the limit dynamical system of $G$ is an example of a \emph{hyperbolic duality}, see~\cite{nek:hyperbolic}.  One of its aspects is that large-scale properties of the action of $G$ on $\xo$ are related to the topological properties of the limit space. Third result of the paper (Theorem~\ref{th:asdym}) is an example of this philosophy: we show that the topological dimension of the limit space is equal to the asymptotic dimension of the orbits of the action of $G$ on $\xo$ (seen as graphs with respect to a finite generating set of $G$). It is analogous to the result of~\cite{BuyaloLebedeva} that topological dimension of the boundary of a Gromov-hyperbolic group is equal to the asymptotic dimension of the group minus one. The analog of a hyperbolic group in the context of contracting self-similar group $G$ is the groupoid of generated by the germs of the action of $G$ on the boundary $\xo$ of the tree $\xs$ and by the germs of the one-sided shift. (The shift is responsible for absence of ``minus one'' in our result compared to the case of hyperbolic groups.) 

Theorem~\ref{th:contractingmodeldim} implies that the condition of being contracting for a self-similar finitely generated group $G$ follows from finitely many relations, even when the group is not finitely presented. This was shown earlier in~\cite{nek:book}, but Theorem~\ref{th:contractingmodeldim} can be seen as a stronger version of this result. We call the corresponding finitely presented contracting group (such that $G$ is its quotient by the kernel of the action on the rooted tree $\xs$) a \emph{finitely presented contracting overgroup} of $G$.

A direct corollary of Theorem~\ref{th:contractingmodeldim} is the fact that in the case of a one-dimensional limit space the contracting overgroup can be chosen to be virtually free. We show the converse in Theorem~\ref{th:onedimdescription}, thus proving that the limit space is one-dimensional if and only if the group has a virtually free contracting overgroup. 

A corollary of Theorem~\ref{th:onedimdescription} is a complete description of contracting faithful \emph{finitely presented} groups with locally connected one-dimensional limit space: they are equivalent to iterated monodromy groups of $z^d$ for $|d|\ge 2$, or Chebyshev polynomials $\pm T_d$.

The only known examples of faithfully acting finitely presented contracting groups are the virtually nilpotent ones (iterated monodromy groups of expanding self-coverings of infra-nil-orbifolds, see~\cite{nek:book}).

The last section of the article presents some examples of contracting groups with one-dimensional limit space. It is interesting that all classical and well-studied self-similar groups (the Grigorchuk group~\cite{grigorchuk:80_en}, Gupta-Sidki group~\cite{gupta-sidkigroup}, Fabrikowski-Gupta group~\cite{gufabr}, Brunner-Sidki-Vieira group~\cite{bsv:jns}, Basilica group~\cite{zukgrigorchuk:3st}, e.t.c.) have one-dimensional limit space. Some higher-dimensional examples where studied in~\cite{nek:dendrites,nek:paperfold}.

The author acknowledges the support of the Institut Henri Poincar\'e and its special trimester ``Groups Acting on Fractals, Hyperbolicity and Self-similarity,'' where the paper was finished, and the NSF grant DMS2204379.

\section{Basic definitions and results}

We present in this section an introduction to the theory of self-similar contracting groups, in particular in order to fix the notations. For more, including proofs, see~\cite{nek:book,nek:dyngroups}.

\subsection{Self-similar groups}

\begin{defi}
A \emph{self-similar group} $(G, \bim)$ is a group $G$ together with a set $\bim$ on which $G$ acts by commuting left and right actions such that the right action is free and has a finite number of orbits. The set $\bim$ is called then the \emph{biset} of the self-similar group.
\end{defi}

Recall, that an action is called free if for any point $x$ and a group element $g$ an equality $x\cdot g=x$ implies that $g$ is the identity.

Two $G$-bisets $\bim_1, \bim_2$ are \emph{isomorphic} if there exists a bijection $F:\bim_1\arr\bim_2$ such that $F(g_1\cdot x\cdot g_2)=g_1\cdot F(x)\cdot g_2$ for all $x\in\bim_1$ and $g_1, g_2\in G$.

Two self-similar groups $(G_i, \bim_i)$, $i=1, 2$, are \emph{equivalent} if there exists a group isomorphism $\phi:G_1\arr G_2$ and a bijection $F:\bim_1\arr\bim_2$ such that $F(g_1\cdot x\cdot g_2)=\phi(g_1)\cdot F(x)\cdot \phi(g_2)$ for all $g_1, g_2\in G_1$ and $x\in\bim_1$.

A \emph{basis} of a biset $\bim$ is a set $\alb\subset\bim$ intersecting every right $G$-orbit exactly once. Then every element $y\in\bim$ can be written in a unique way as $y=x\cdot g$ for some $x\in\alb$ and $g\in G$.

Let $\alb\subset\bim$ be a basis. Then for every $g\in G$ there exists a permutation $\sigma_g\in\symm(\alb)$ and a map $\alb\arr G:x\mapsto g|_x$ such that
\[g\cdot x=\sigma_g(x)\cdot g|_x\]
for every $x\in\alb$. It is easy to check that the map $g\mapsto (\sigma_g, (g|_x)_{x\in\alb})$ is a homomorphism from $G$ to the wreath product $\symm(\alb)\ltimes G^{\alb}$. We call it the \emph{wreath recursion} associated with the biset and its basis. Every homomorphism from $G$ to the wreath product $\symm(\alb)\ltimes G^{\alb}$ arises in this way and thus defines a self-similar group. Changing the basis corresponds to composing the wreath recursion with an inner automorphism of the wreath product.

If $\bim_1, \bim_2$ are two $G$-bisets (i.e., sets with commuting right and left actions of $G$), then we denote by $\bim_1\otimes\bim_2$ the direct product $\bim_1\times\bim_2$ modulo the left action of $G$ given by $(x_1, x_2)\mapsto (x_1\cdot g^{-1}, g\cdot x_2)$. We write the class (the orbit) of an element $(x_1, x_2)$ by $x_1\otimes x_2$. The set $\bim_1\otimes\bim_2$ is naturally a biset with respect to the actions $g_1\cdot (x_1\otimes x_2)\cdot g_2=(g_1\cdot x_1)\otimes (x_2\cdot g_2)$. It is also easy to show that the bisets $(\bim_1\otimes\bim_2)\otimes\bim_3$ and $\bim_1\otimes(\bim_2\otimes\bim_3)$ are isomorphic.

In particular, if $(G, \bim)$ is a self-simiar group, then for every natural number $n$ we have a biset $\bim^{\otimes n}$, i.e., a self-similar group $(G, \bim^{\otimes n})$. We also define $\bim^{\otimes 0}=G$ with the natural left and right $G$-actions.

If $\alb$ is a basis of $\bim$, then $\alb^n=\{x_1\otimes x_2\otimes\cdots\otimes x_n\;:\;x_i\in\alb\}$ is a basis of $\bim^{\otimes n}$. We get therefore the associated wreath recursion $G\arr\symm(\alb^n)\ltimes G^{\alb^n}$, and hence an action of $G$ on $\alb^n$ (via the homomorphism from $\symm(\alb^n)\ltimes G^{\alb^n}$ to $\symm(\alb^n)$). These actions, taken together, define an action of $G$ on the \emph{rooted tree} $\xs=\bigcup_{n\ge 0}\alb^n$. We consider the set of finite words $\xs$ to be a tree with the root equal to the empty word (the basis of $\bim^{\otimes 0}$) in which a word $v$ is connected by an edge to the words of the form $vx$ for $x\in\alb$.

One can define the tree $\xs$ and the $G$ action on it also without choosing a basis. Namely, the tree $\xs$ is isomorphic to the tree with the set of vertices equal to the set of orbits $\bim^*/G$ of the right action of $G$ on the disjoint union $\bim^*$ of the bisets $\bim^{\otimes n}$ for all $n\ge 0$. The left action of $G$ on $\bim^*$ induces then an action on $\bim^*/G$. The tree structure is induced from the right divisibility relation on the semigroup $\bim^*$: the orbit of $v\in\bim^{\otimes n}$ is connected to the orbits of elements of the form $v\otimes x\in\bim^{\otimes n}$ for $x\in\bim$.

We will denote by $g(v)$, for $v\in\xs$ and $g\in G$, the image of $v$ under the action of $G$ on $\xs$. We also denote by $g|_v$ the unique element of $G$ such that 
\[g\cdot v=g(v)\cdot g|_v.\]

The elements $g|_v$ are called \emph{sections} of $g$. (They depend on the choice of $\alb$.) They satisfy the following conditions:
\[g|_{v_1v_2}=(g|_{v_1})|_{v_2},\qquad (g_1g_2)|_v=g_1|_{g_2(v)}g_2|_v.\]

For a given $x\in\bim$, the \emph{associated virtual endomorphism} $\phi_x$ is the homomorphism given by the condition
\[g\cdot x=x\cdot \phi_x(g)\]
defined on the subgroup of elements of $G$ such that $g\cdot x$ and $x$ belong to the same right $G$-orbit. If $x$ is an element of a basis $\alb$, then the domain of $\phi_x$ is the stabilizer of the vertex $x$ of the rooted tree $\xs$, and the virtual endomorphism is the section map $g\mapsto g|_x$. The virtual endomorphism $\phi_x$ (unlike the section map) depends only on $x$.

\begin{defi}
The \emph{faithful quotient} of a self-similar group $(G, \bim)$ is the group $(G/K, \bim/K)$, where $K\lhd G$ is the kernel the action of $G$ on the tree $\xs$, and $\bim/K$ is the set of right $K$-orbits with the natural structure of a $G/K$-biset.
\end{defi}

\begin{defi}
We say that $(G, \bim)$ is \emph{level-transitive} if the action of $G$ on each level of the associated rooted tree is transitive. i.e., if the left action of $G$ on the set of right orbits $\bim^{\otimes n}/G$ is transitive for every $n$.

We say that a self-similar group $(G, \bim)$ is \emph{self-replicating} if the left action of $G$ on $\bim$ is transitive. Equivalently, it is self-replicating if the action of $G$ on the first level of the tree $\xs$ is transitive, and for any $x\in\alb$ the virtual endomorphism $g\mapsto g|_x$ from the stabilizer of $x$ to $G$ is surjective.
\end{defi}

Every self-replicating self-similar group is level-transitive.

\begin{example}
Consider the wreath recursion $a\mapsto\sigma(1, b), b\mapsto (1, a)$ over the free group generated by $a$ and $b$. Here $\sigma$ is the non-trivial element $(01)$ of the symmetric group of permutations of the alphabet $\{0, 1\}$. We do not write trivial elements of the symmetric groups (for example, $b$ acts trivially on the alphabet). We denote by $1$ the trivial element of the self-similar group. We have $a^2\mapsto (b, b)$, $b\mapsto (1, a)$, $a^{-1}ba\mapsto (a, 1)$. It follows that the group is self-replicating.

The faithful quotient of this self-similar group is not free. For example, the elements $b$ and $a^{-1}ba$ commute in the faithful quotient, since  the commutator is contained in the kernel of the wreath recursion. We usually write $a=\sigma(1, b)$, $b=(1, a)$ when defining the faithful quotient, thus identifying the group with its image in the wreath product. (The wreath recursion is injective for faithful self-similar groups.)
\end{example}

\subsection{Contracting groups and their models}

\begin{defi}
A self-similar group $(G, \bim)$ is said to be \emph{contracting} if for some basis $\alb\subset\bim$ there exists a finite set $\nuke\subset G$ such that for every $g\in G$ there exists $n$ such that $g|_v\in\nuke$ for all $v\in\xs$ of length at least $n$.
\end{defi}

One can show, see~\cite{nek:book}, that if a self-similar group is contracting with respect to some basis $\alb$, then it is contracting with respect to every basis.
The smallest set $\nuke$ satisfying the above definition is called the \emph{nucleus} of the group (it depends on the choice of the basis $\alb$).

It is also shown that if $G$ is finitely generated, then it is contracting if and only if there exist $C>1$ and $\lambda\in (0, 1)$ such that for all $g\in G$ and all long enough words $v\in\xs$ we have
\[\|g|_v\|\le C\lambda^{|v|}\|g\|,\]
where $\|g\|$ denotes the length of $g$ with respect to a fixed finite generating set of $G$.

Suppose that $G$ is generated by a finite set $S$. Then a finite set $N\ni 1$ contains the nucleus if and only if there exists $n$ such that for every $g\in N, s\in S, v\in\alb^n$ the section $(gs)|_v$ belongs to $N$. (We can also replace $(gs)|_v$ by $(sg)|_v$ in this condition.)

Let $(G, \bim)$ be a contracting self-similar group, and suppose that $\X$ is a topological space on which $G$ acts (from the right) properly and co-compactly by homeomorphisms.

Define $\X\otimes\bim$ as the quotient of the direct product $\X\times\bim$ by the action $g:(\xi, x)\mapsto (\xi\cdot g^{-1}, g\cdot x)$ of $G$, where $\bim$ has discrete topology. Then $\X\otimes\bim$ is also a proper co-compact right $G$-space.

We are interested in $G$-equivariant continuous maps $I:\X\otimes\bim\arr\X$.

Any such a map is uniquely determined by the maps $I_x:\X\arr\X$, for $x\in\alb$, given by $I_x(\xi)=I(\xi\otimes x)$.
The maps $I_x$ agree with the associated wreath recursion in the following way:
\begin{equation}
\label{eq:Ixcompatible}
I_x(\xi\cdot g)=I_{g(x)}(\xi)\cdot g|_x
\end{equation}
for all $x\in\alb$, $g\in G$, $\xi\in\X$.

Conversely, any collection of continouous maps $I_x:\X\arr\X$ satisfying the above condition defines an equivariant map $I:\X\otimes\bim\arr\X$ by the condition $I(\xi\otimes x\cdot g)=I_x(\xi)\cdot g$ for $x\in\alb$ and $g\in G$.

Let $I:\X\otimes\bim\arr\X$ be a $G$-equivariant map. Denote by $\M$ and $\M_1$ the spaces of orbits $\X/G$ and $(\X\otimes\bim)/G$, respectively. Then the correspondences:
\[\xi\otimes x\mapsto\xi\]
and
\[\xi\otimes x\mapsto I(\xi\otimes x)\]
are well defined continuous maps $\M_1\arr\M$, which we will denote by $f$ and $\iota$, respectively. 

\begin{defi}
A \emph{covering correspondence} is a pair of continuous maps $f, \iota:\M_1\arr\M$, where $f$ is a finite degree covering. We call $f$ and $\iota$ the \emph{covering} and the \emph{reduction maps}, respectively.
\end{defi}

If $f, \iota:\M_1\arr\M$ is a covering correspondence and $\M$ is path-connected, then we can construct a natural biset over the fundamental group $\pi_1(\M)$ in the following way. Choose a basepoint $t\in\M$, and define $\bim$ as the set of pairs $(z, [\ell])$, were $z\in f^{-1}(t)$, and $\ell$ is the homotopy class of a path $\ell$ in $\M$ starting in $t$ and ending in $\iota(z)$. 

The right action of $\pi_1(\M, t)$ on $\bim$ is defined just by appending loops: $(z, [\ell])\cdot [\gamma]=(z, [\ell\gamma])$, where we multiply paths in an ``unnatural way'' as functions ($\gamma$ is passed before $\ell$ in the concatenation $\ell\gamma$).

The left action is defined by lifting loops by the covering $f$ and then mapping them back to $\M$ by the reduction map $\iota$:
\[[\gamma]\cdot (z_1, [\ell])=(z_2, [\iota(\gamma_{z_1})\ell]),\]
where $\gamma_{z_1}$ is the lift of $\gamma$ by $f$ starting in $z_1$, and $z_2$ is the end of $\gamma_{z_1}$.

It is easy to check that $(\pi_1(\M, t), \bim)$ is a self-similar group. The associated virtual endomorphism is $\iota_*\circ f^*$ (defined up to inner automorphisms of the fundamental group).

\begin{defi}
The faithful quotient of $(\pi_1(\M, t), \bim)$ is called the \emph{iterated monodromy group} of the correspondence $f, \iota:\M_1\arr\M$.
\end{defi}

A particular case of the above definition is when the reduction map $\iota$ is the identity homeomorphism or an identical embedding. Then we do not mention the map $\iota$ and talk about the iterated monodromy group of the (partial) self-covering $f:\M_1\arr\M$ for $\M_1\subseteq\M$.

Let $f, \iota:\M_1\arr\M$ be a covering correspondence.
Suppose that $\X$ is the universal covering of $\M$. It can be seen as the space of homotopy classes of paths starting in $t$. The right action of $G$ on $\X$ is the usual one: by appending the loops at the beginning of the paths. It is checked directly then that the map
\[[\delta]\otimes (z, [\ell])\mapsto \iota(\delta_z)\ell\]
is a well defined equivariant map from $\X\otimes\bim$ to $\X$. Here $\delta_z$ is the lift of $\delta$ by $f$ to a path starting in $z$. Moreover, the corresponding maps $f, \iota:\X\otimes\bim\arr\X$, defined using this equivariant map (see above), coincide with the original maps $f, \iota:\M_1\arr\M$.

The general situation (of non-free actions of $G$ on $\X$) is interpreted in a similar way using the notion of \emph{orbispaces}, see~\cite{nek:book,nek:dyngroups}.

Suppose that $(G, \bim)$ is a self-similar group, and let $I:\X\otimes\bim\arr\X$ be a $G$-equivariant map, where $\X$ is a right proper co-compact $G$-space.

We define the spaces $\X\otimes\bim^{\otimes n}$ and the corresponding maps
\[I_n:\X\otimes\bim^{\otimes (n+1)}\arr\X\otimes\bim^{\otimes n}\]
inductively by $I_n(\xi\otimes x_1\otimes x_2\otimes\cdots\otimes x_n)=I(\xi\otimes x_1)\otimes x_2\otimes\cdots\otimes x_n$.

We denote by $I^n:\X\otimes\bim^{\otimes n}\arr\X$ the composition $I_{n-1}\circ\cdots\circ I_1\circ I$.

In the case of a $\pi_1(\M)$-equivariant map $I:\X\otimes\bim\arr\X$ associated with a covering correspondence $f, \iota:\M_1\arr\M$, the spaces of orbits $\M_n=(\X\otimes\bim^{\otimes n})/G$ are naturally homeomorphic to the space of sequences $(x_1, x_2, \ldots, x_n)\in\M_1^n$ satisfying $f(x_i)=\iota(x_{i+1})$. 

In particular, if $f:\M_1\arr\M$ is a partial self-covering (i.e., $\iota:\M_1\arr\M$ is the identical embedding of a subset $\M_1\subseteq\M$), then $\M_n$ is the domain of the $n$th iteration of $f$.

\subsection{Contracting models}

Suppose now that $\X$ is a metric space on which $G$ acts properly from the right by isometries. We allow the metric to take infinite value.

\begin{defi}
We say that a $G$-equivariant map $I:\X\otimes\bim\arr\X$ is \emph{contracting} if there exist $L>1$ and $n\ge 1$ such that for all $v\in\bim^{\otimes n}$ and $\xi_1, \xi_2\in\X$ we have
\[d(I^n(\xi_1\otimes v), I^n(\xi_2\otimes v))\le L^{-1}d(\xi_1, \xi_2).\]
\end{defi}

Since the action of $G$ on $\X$ is by isometries, it is enough to check the contraction conditions for the maps $\xi\mapsto I^n(\xi\otimes v)$, for $v\in\alb^n$, only.

\begin{example}
Let $I:\X\otimes\bim\arr\X$ be the $\pi_1(\M)$-equivariant map associated with a covering correspondence $f, \iota:\M_1\arr\M$, where $\M$ and $\M_1$ are  compact metric space, $f$ is a local isometry, and $\iota$ is a locally contracting map, as above. Then $I$ is contracting. In particular, if $\iota$ is the identity homeomorphism, and $f$ is expanding with respect to the same metric on $\M$ and $\M_1$ (identified with each other by $\iota$), then the corresponding map $I$ is contracting.
\end{example}

The following theorems are proved in~\cite{nek:models,nek:dyngroups}. 

\begin{theorem}
\label{th:limgiscontracting}
Let $(G, \bim)$ be a contracting finitely generated group. Then there exists a metric space $\limg$, a proper co-compact right action of $G$ on it by isometries, and a contracting $G$-equivariant \emph{homeomorphism} $I:\limg\otimes\bim\arr\limg$. Moreover, such space is unique: if $\X'$ and $I'$ is another space and a homeomorphism, then there exists a $G$-equivariant homeomorphism $\Phi:\limg\arr\X'$ such that $\Phi(I(\xi\otimes x))=I'(\Phi(\xi)\otimes x)$ for all $\xi\in\limg$ and $x\in\bim$. We call this space the \emph{limit $G$-space} of $(G, \bim)$.
\end{theorem}

Let us show how the metric on $\limg$ is defined, since the metric given in~\cite[Proposition~4.5.30]{nek:dyngroups} is defined only locally in the case when $\limg$ is not connected (i.e., it may be infinite between points in different connected components). Let $\ell(\xi_1, \xi_2)$ be the function $\limg\times\limg\to\{-\infty\}\cup\mathbb{N}$ given by the condition that $\ell(\xi_1, \xi_2)$ is the largest $n$ such that there exist tiles $\til\otimes v_1, \til\otimes v_2$ of level $n$ such that $\xi_i\otimes v_i\ni\xi_i$ and $\til\otimes v_1\cap\til\otimes v_2\ne\emptyset$. If such tiles do not exist, then we set $\ell(\xi_1, \xi_2)=-\infty$. Note that for any two points $\xi_1, \xi_2\in\limg$ there exists $n$ such that $\ell(\xi_1\otimes v, \xi_2\otimes v)>-\infty$ for all $v\in\bim^{\otimes n}$.

Fix also a word metric $d_S$ on $G$. 
Let $\alpha>0$ and let \[d_\alpha(\xi_1, \xi_2)=\min\{e^{-\alpha\ell(\xi_1, \xi_2)}, d_S(g_1, g_2)\},\] where $g_1, g_2$ are elements of $G$ such that $\xi_i\in\til\cdot g_i$ (we can choose an arbitrary pair of such elements for every pair of points). Then, for all $\alpha$ small enough, the function 
\[d(\xi, \zeta)=\min_{\eta_0=\xi, \eta_1, \ldots, \eta_n=\zeta}\sum_{i=0}^{n-1}d_\alpha(\eta_i, \eta_{i+1})\]
is a metric satisfying the conditions of Theorem~\ref{th:limgiscontracting}. 

\begin{theorem}
\label{th:modelGspace}
If $I:\X\otimes\bim\arr\X$ is an arbitrary contracting equivariant map, then the inverse limit of the spaces $\X\otimes\bim^{\otimes n}$ with respect to the maps $I_n$ is homeomorphic to the limit $G$-space $\limg$. Moreover, the homeomorphism of $\lim_{\longleftarrow}\X\otimes\bim^{\otimes n}$ with $\limg$ conjugates the limit of the natural homeomorphisms $\X\otimes\bim^{\otimes n}\otimes\bim\arr\X\otimes \bim^{\otimes (n+1)}$ with the canonical homemorphism $\limg\otimes\bim\arr\limg$.
\end{theorem}

In view of the above theorem, we call a contracting map $I:\X\otimes\bim\arr\X$ the \emph{model} of the limit $G$-space or the \emph{contracting model} of $(G, \bim)$.

The limit $G$-space $\limg$ can be constructed synthetically in the following way. Choose a basis $\alb$ of $\bim$, and let $\xmo$ be the space of left-infinite sequences $\ldots x_2x_1$ of elements of $\alb$ with the direct product topology. Consider the direct product $\xmo\times G$, where $G$ is discrete. We say that $\ldots x_2x_1\cdot g$ and $\ldots y_2y_1\cdot h$ are \emph{asymptotically equivalent}, if there exists a sequence $g_n$ of elements of a finite subset of $G$ such that 
\[g_n\cdot x_n\ldots x_2x_1\cdot g=y_n\ldots y_2y_1\cdot h\]
in $\bim^{\otimes n}$ for every $n$. Equivalently, we may assume that $g_n$ belong to the nucleus and $g_n\cdot x_n=y_n\cdot g_{n-1}$, for all $n\ge 1$, while $g_0g=h$.

The quotient of $\xmo\times G$ by this equivalence relation is the limit $G$-space $\limg$. The homeomorphism $I:\limg\otimes\bim\arr\bim$ is induced by the natural map
\[\ldots x_2x_1\cdot g\otimes x=\ldots x_2x_1y\cdot h\]
for $y\in\alb$ and $h\in G$ such that $g\cdot x=y\cdot h$. In view of this identification, we will usually omit $I$, identify $\limg\otimes\bim$ with $\limg$, and write $I(\xi\otimes x)$ just as $\xi\otimes x$ (but we will do it only for the limit $G$-space $\limg$).

We will denote the space of orbits $\limg/G$ by $\lims$. It follows from the above description of $\limg$ that $\lims$ can be defined as the quotient of $\xmo$ by the relation identifying $\ldots x_2x_1$ and $\ldots y_2y_1$ if there exists a sequence $g_n$ of elements of the nucleus such that $g_n(x_n\ldots x_2x_1)=y_n\ldots y_2y_1$ for all $n$.

The image of $\xmo\cdot 1\subset\xmo\times G$ in $\limg$ is denoted $\til$ and is called the \emph{tile} of the limit $G$-space. It is a compact set depending on the choice of the basis $\alb$. We have $\limg=\bigcup_{g\in G}\til\cdot g$.

The following theorem is proved in~\cite{nek:book}.

\begin{theorem}
\label{th:connectivity}
Let $(G, \bim)$ be a finitely generated contracting group.
The space $\lims$ is connected if and only if $(G, \bim)$ is level-transitive. The space $\limg$ is connected if and only if $(G, \bim)$ is self-replicating. If $(G, \bim)$ is self-replicating, then $\limg$ is locally connected.
\end{theorem}

Let $(G, \bim)$ be a contracting group, and let $\alb$ be a basis of $\bim$. Suppose that $S$ is a finite symmetric generating set containing the nucleus and such that $g|_x\in S$ for every $g\in S$ and $x\in\alb$. Note that any finite generating set of a contracting group is a subset of a generating set satisfying these conditions. Consider the corresponding \emph{Cayley simplicial complex} with the set of vertices $G$ in which a subset $A\subset G$ is a simplex if and only if $gh^{-1}\in S$ for all $g, h\in A$. Let $\Delta_S$ be its geometric realization, where each $d$-dimensional simplex is realized as the standard affine simplex $\{(x_0, x_1, \ldots, x_d)\in\R^{d+1}\;:\;x_i\ge 0, \sum_ix_i=1\}$. We consider it with the $\ell^1$-metric, so that distance between $(x_0, x_1, \ldots, x_d)$ and $(y_0, y_1, \ldots, y_d)$ is $\sum_{i=0}^d|x_i-y_i|$.

The group $G$ acts naturally from the right on the Cayley complex, and hence it acts by isometries on $\Delta_S$. The action is obviously proper (as it is free on the set of vertices) and co-compact (the set of vertices is one orbit). We also have natural maps $I_x$ defined on the vertices by
\[I_x(g)=g|_x.\]
The maps $I_x$ are simplicial (since $S|_\alb\subset S$), and hence induce continous maps on $\Delta_S$. They satisfy the compatibility condition~\eqref{eq:Ixcompatible}, hence define a $G$-equivariant map $I:\Delta_S\otimes\bim\arr\Delta_S$.
Its $n$th iteration $I^n:\Delta_S\otimes\bim^{\otimes n}\arr\Delta_S$ is defined on the set of vertices by
\[I_v(g)=g|_v.\]

The following theorem was proved in~\cite{nek:models}.

\begin{theorem}
For all sufficiently large $m$ the $m$th iteration $I^m:\Delta_S\otimes\bim^{\otimes m}\arr\Delta_S$ is $G$-equivariantly homotopic to a contracting map.
\end{theorem}

This theorem implies that the limit space $\limg$ is an inverse limit of the simplicial complexes $\Delta_S\otimes\bim^{\otimes n}$.

\subsection{Topological dimension}

In this subsection, we give a short overview of definitions and basic facts about the topological dimension.

A \emph{cover} of a set $\X$ is a collection $\mathcal{U}$ of sets such that $\X\subseteq\bigcup_{A\in\mathcal{U}}A$. We also use notation $\bigcup\mathcal{U}=\bigcup_{A\in\mathcal{U}}A$.

If $\mathcal{U}$ and $\mathcal{V}$ are covers of a set $\X$, then we say that $\mathcal{V}$ is a \emph{refinement of} $\mathcal{U}$ if for every $V\in\mathcal{V}$ there exists $U\in\mathcal{U}$ such that $V\subset U$.

We say that a cover $\mathcal{U}$ of a set $X$ is of \emph{order} $n$ if every point of $X$ is contained in at most $n$ elements of $\mathcal{U}$.

\begin{defi}
We say that a topological space $\X$ has \emph{(Lebesgue) covering dimension} $\le d$ if for every open cover $\mathcal{U}$ of $\X$ there exists an open cover $\mathcal{V}$ of order $d+1$ refining $\mathcal{U}$.
\end{defi}

We can also use closed covers in the case of a metric space, since multiplicity of a closed cover is equal to the muliplicity of the cover by sufficiently small open neighborhoods of the covering sets.

If the space $\X$ is a separable metric space, then the covering dimension is also equal to (both versions of) the inductive dimension of $\X$, so we will just call it dimension, and denote it by $\dim\X$.

The following description of the covering dimension follows from Ostrand's theorem~\cite{ostrand:dimension}, see 
also~\cite[Proposition~1.6]{KirchbergWinter}.

\begin{proposition}
Let $\X$ be a metrizable space.
The covering dimension of $\X$ is at most $d$ if and only if for every open (resp.\  closed) cover $\mathcal{U}$ of $\X$ there exists an open (resp.\ closed) refinement $\mathcal{V}$ of $\mathcal{U}$ and a partition $\mathcal{V}=\mathcal{V}_0\cup\mathcal{V}_1\cup\ldots\cup\mathcal{V}_d$ such that for every $i=0, 1, \ldots, d$ the elements of $\mathcal{V}_i$ are pairwise disjoint.
\end{proposition}

\begin{defi}
Let $\limg$ be the limit $G$-space of a contracting group $(G, \bim)$. We say that a set $U$ is \emph{$G$-adapted} if it has compact closure and for every $g\in G$ we have either $U\cdot g=U$ or the closures of $U\cdot g$ and $U$ are disjoint. 
\end{defi}

\begin{lemma}
\label{lem:adapted}
For every neighborhood $W$ of a point $\xi\in\limg$ there exists a $G$-adapted open neighborhood $U$ of $\xi$ such that $U\subset W$. 
\end{lemma}

\begin{proof}
The space $\limg$ is locally compact, so there exists a compact neighborhood $U_1$ of $\xi$. Since the action of $G$ on $\limg$ is proper, the set $\{g_1, g_2, \ldots, g_n\}$ of elements $g\in G$ such that $U_1\cap U_1\cdot g\ne\emptyset$ is finite. For every $g_i$ such that $\xi\cdot g_i\ne \xi$, find disjoint neighborhoods $V_i\subset W$ of $\xi$ and $W_i$ of $\xi\cdot g_i$. Then $V_i'=V_i\cap W_i\cdot g_i^{-1}$ is a neighborhood of $\xi$ such that $V_i'\subset V_i$ and $V_i'\cdot g_i\subset W_i$ are disjoint. Take the intersection $U_2$ of all such neighborhoods $V_i$. Then $U_2\cdot g$ and $U_2$ are disjoint for all $g\in G$ such that $\xi\cdot g\ne\xi$.

The stabilizer of $\xi$ in $G$ is finite. It follows that the intersection $U$ of the neighborhoods $U_2\cdot g$ for all elements $g\in G$ such that $\xi\cdot g=\xi$ is a neighborhood of $\xi$. Consequently, $U$ is a $G$-adapted neighborhood of $\xi$.
\end{proof}

\begin{proposition}
\label{prop:dimlimglims}
Let $(G, \bim)$ be a contracting group. Then $\dim\limg=\dim\lims$.
\end{proposition}

\begin{proof}
Let us show that the quotient map $\limg\arr\lims=\limg/G$ is open. Let $A$ be an open subset of $\limg$, and let $\xi\in A$. It is enough to show that there exists a neighborhood of $\xi$ contained in $A$ and mapped by the quotient map to an open subset of $\lims$.

There exists a $G$-adapted open neighborhood $U$ of $\xi$ such that $U\subset A$. Denote by $\overline U$ its image in $\lims$. The set $\bigcup_{g\in G}U\cdot g$ is the full preimage in $\limg$ of $\overline U$. The space $\lims$ is, by definition, the quotient of the topological space $\limg$ by the group action. Since the full preimage of $\overline U$ is open in $\limg$, the set $\overline U$ is open.

The quotient map $\limg\arr\lims$ is countable-to-one, since we assume that $G$ is countable. The statement about topological dimension follows then from a result of P.~Alexandroff, see~\cite[Theorem~1.12.8]{engel:dimension}. 
\end{proof}

We say that a cover $\mathcal{U}$ of a metric space $X$ is \emph{uniformly bounded} if $\sup_{U\in\mathcal{U}}\diam U<\infty$.

The following notion is due to M.~Gromov~\cite{gro:asympt}.

\begin{defi}
Let $\X$ be a metric space. We say that its asymptotic dimension is at most $d$ if for every $R\ge 0$ there exists a uniformly bounded cover $\mathcal{U}$ such that every set of diameter at most $R$ intersects at most $d+1$ elements of $\mathcal{U}$.
\end{defi}

An equivalent definition is given by the following proposition.

\begin{proposition}
A metric space $\X$ has asymptotic dimension at most $d$ if and only if for every $R>0$ there exists a uniformly bounded cover $\mathcal{U}$ which can be represented as a union $\mathcal{U}=\mathcal{U}_0\cup\mathcal{U}_1\cup\ldots\cup\mathcal{U}_d$ of $d+1$ sets such that for every $i=0, 1, \ldots, d$ the distance between any two different elements of $\mathcal{U}_i$ is greater than $R$.
\end{proposition}

\section{Topological dimension of the limit space}

\subsection{Minimal dimension of a contracting model}

The following theorem is proved in~\cite{nek:dyngroups}. We give here the proof for completeness, and also since parts of it will be used later in our paper.

\begin{theorem}
\label{th:contractingmodeldim}
Let $(G, \bim)$ be a contracting self-similar group. Then the following conditions are equivalent.
\begin{enumerate}
\item $\dim\limg\le d$.
\item There exists $m\ge 1$, a $d$-dimensional simplicial complex $\Gamma$ with proper co-compact action of $G$ on it, and a contracting $G$-equivariant piecewise affine map $I:\Gamma\otimes\bim^{\otimes m}\arr\Gamma$. 
\end{enumerate}
If $G$ is finitely generated and $d\ge 1$, then we may assume that the complex $\Gamma$ is connected.
\end{theorem}

In other words the topological dimension of the limit space is equal to the minimal dimension of a contracting affine model of the group.

\begin{proof}
Let us prove that (2) implies (1).
If $I:\Gamma\otimes\bim\arr\Gamma$ is a contracting model, then the limit space is homeomorphic to the inverse limit of the spaces $\Gamma\otimes\bim^{\otimes n}$ with respect to the maps $I_n$, see Theorem~\ref{th:modelGspace}. Each of the complexes $\Gamma\otimes\bim^{\otimes n}$ has the same dimension as $\Gamma$. It is known that the topological dimension of an inverse limit is not greater than the maximal topological dimension of the terms sequence, see~\cite[Proposition~2.3]{pearse:dimension}. Consequently, $\dim\limg$ is not greater than the dimension of $\Gamma$.

Let us prove the converse statement.
Let $\mathcal{W}$ be an arbitrary $G$-invariant cover of $\limg$ by open $G$-adapted sets.
For every $U\in\mathcal{W}$, the $G$-orbit of $U$ is the full preimage of the image of $U$ in $\lims$. It follows that the image of $U$ in $\lims$ is open. We get an open cover of $\lims$ by the images of the elements of $\mathcal{W}$. Consequently, by the Lebesgue covering lemma, for every sufficiently small subset $A$ of $\lims$ there exists $U\in\mathcal{W}$ such that $A$ is contained in the image of $U$. Let $\overline A$ be the full preimage of $A$ in $U$. Then $\overline A$ is $G$-adapted.

This implies that for every cover $\mathcal{U}_1$ of $\lims$ by sufficiently small subsets we can find a $G$-invariant cover $\mathcal{U}$ of $\limg$ by adapted sets such the set of images of the elements of $\mathcal{U}$ in $\lims$ is $\mathcal{U}_1$. We say that $\mathcal{U}$ is the \emph{lift} of $\mathcal{U}_1$.

Suppose that $\mathcal{V}$ is a cover of $\lims$  of multiplicity $\le d+1$ by sufficiently small \emph{closed} sets. By the above arguments, it can be lifted to a $G$-invariant cover $\mathcal{U}$ of $\lims$ by adapted sets. Suppose that $U_1, U_2, \ldots U_k$ are pairwise different elements of $\mathcal{U}$ such that $U_1\cap U_2\cap\ldots\cap U_k\ne\emptyset$. Then, by the definition of a $G$-adapted set, the images of $U_i$ in $\lims$ are pairwise different, hence $k\le d+1$. Consequently, the multiplicity of $\mathcal{U}$ is also $\le d+1$. Let us fix this cover $\mathcal{U}$.

\begin{lemma}
\label{lem:contractioncompact}
Let $\mathcal{C}\subset\limg$ be compact. There exists $m$ such that for every $v\in\bim^{\otimes n}$ such that $n\ge m$ the set of elements $U\in\mathcal{U}$ intersecting $\mathcal{C}\otimes v$ has a non-empty intersection.
\end{lemma}

\begin{proof}
Since $\mathcal{U}$ is $G$-invariant, it is enough to prove the statement for $v\in\alb^n$. The union of the sets $\mathcal{C}\otimes v$ for $v\in\xs$ has compact closure, since the maps $\xi\mapsto\xi\otimes x$ are contractions. It follows that the set of elements of $\mathcal{U}$ intersecting $\mathcal{C}\otimes v$ for some $v\in\xs$ is finite. Let us denote it $\mathcal{U}_0$. Suppose that the statement of the lemma is not true. Then there exists a strictly increasing sequence $n_k$ and a sequence $v_k\in\alb^{n_k}$ such that the intersection of the set of elements of $\mathcal{U}_0$ intersecting $\mathcal{C}\otimes v_k$ is empty. Since the number of subsets of $\mathcal{U}_0$ is finite, after passing to a subsequence, we may assume that the set $\mathcal{U}_1$ of elements of $\mathcal{U}$ intersecting $\mathcal{C}\otimes v_k$ is constant. Choose for every $U\in\mathcal{U}_1$ and every $k$ a point $\xi_{U, k}\in\mathcal{C}$ such that $\xi_k\otimes v_k\in U$. Passing to a sub-sequence, we may assume that $\xi_{U, k}\otimes v_k$ converges for every $U\in\mathcal{U}_1$. Since the points $\xi_{U, k}$ belong to a compact subset of $\limg$, the limit depends only on the sequence $v_k$. The limit belongs to $U$, since $U$ is closed. But this shows that the limit belongs to the intersection of the elements of $\mathcal{U}_1$, i.e., that the intersection is non-empty, which is a contradiction.
\end{proof}

There exists $m$ such that if $A, B\subset\mathcal{U}$ are finite subsets such that $A$ has non-empty intersection and every element of $B$ intersects $(\bigcup A)\otimes v$ for some $v\in\bim^{\otimes n}$ for $n\ge m$, then the intersection of the elements of $B$ is non-empty. This statement follows from Lemma~\ref{lem:contractioncompact}, since the set of possible subsets $A$ is finite modulo the action of $G$.

Let $\Gamma$ be the nerve of $\mathcal{U}$. It has dimension at most $d$ with a proper action of $G$ on it. Note that if $g\in G$ leaves a simplex of $\Delta$ invariant, then it fixes every its vertex (by the definition of an adapted set).

Define, for a vertex $U\in\mathcal{U}$ of $\Gamma$ and $v\in\bim^{\otimes m}$, the image $I(U\otimes v)$ as the barycenter of the simplex formed by all the elements of $\mathcal{U}$ intersecting $U\otimes v$. It exists by the choice of $m$.

Fix $v\in\bim^{\otimes m}$, and let $A=\{U_1, U_2, \ldots, U_k\}$ be a simplex of $\Gamma$, i.e., a subset of $\mathcal{U}$ with a non-empty intersection. 
Let $B$ be the set of elements of $\mathcal{U}$ intersecting $\bigcup_{i=1}^kU_i\otimes v$. It is a simplex of $\Gamma$. Let $B_i$ be the simplex of elements of $\mathcal{U}$ intersecting $U_i\otimes v$. We have $B_i\subset B$. By definition, $I(U_i\otimes v)$ is the barycenter of $B_i$. Since all simplices $B_i$ are contained in the simplex $B$, we can extend the map $I(\cdot\otimes v)$ in a consistent way to a map from the complex $\Gamma$ to itself, so that it is affine on each of the simplices of $\Gamma$. These extension agree with the biset structure, and we get an equivariant map $I:\Gamma\otimes\bim^{\otimes m}\arr\Gamma$. 

Let us show that the map $I(\cdot\otimes v)$ is contracting for every $v\in\bim^{\otimes m}$. We will use the $\ell^1$ metric for the barycentric coordinates on the simplices. Let $A, B, B_i$ be as above. The map $I(\cdot\otimes v)$ maps a point with barycentric coordinates $(x_1, x_2, \ldots, x_k)$ of the geometric realization of the simplex $A$  to the point of the geometric realization of the simplex $B$ with barycentric coordinates $x_1\vec b_1+x_2\vec b_2+\cdots+x_k\vec b_k$, where $\vec b_k$ is the barycenter of $B_i\subset B$.

Let $U$ be an element of the cover $\mathcal{U}$ intersecting $\bigcap_{i=1}^k U_i$. Suppose that it corresponds to the first barycentric coordinate in the simplex $B$.  We know that $U$ is contained in each simplex $B_i$. It follows that the first coordinate of each vector $\vec b_i$ is $1/|B_i|$. Consequently, the first coordinate of each $\vec b_i$ is greater or equal to $1/|B|$.

Let $\vec x=(x_1, x_2, \ldots, x_k)$ and $\vec y=(y_1, y_2, \ldots, y_k)$ be two points of the geometric realization of $A$. Then the distance between $I(\vec x\otimes v)$ and $I(\vec y\otimes v)$ is equal to the $\ell^1$-norm of the vector $(x_1-y_1)\vec b_1+(x_2-y_2)\vec b_2+\cdots+(x_k-y_k)\vec b_k$. Let $\vec b_k'=\vec b_k-(1/|B|, 0, 0, \ldots, 0)$. Since $\sum_{i=1}^k(x_i-y_i)=0$, we have $(x_1-y_1)\vec b_1+(x_2-y_2)\vec b_2+\cdots+(x_k-y_k)\vec b_k=(x_1-y_1)\vec b_1'+(x_2-y_2)\vec b_2'+\cdots+(x_k-y_k)\vec b_k'$. Let $\vec b_i'=(b_{i1}, b_{i2}, \ldots, b_{im})$.
The coordinates $b_{ij}$ are all non-negative, and $\sum_{j=1}^mb_{ij}=1-\frac{1}{|B|}$.  It follows that the distance between $I(\vec x\otimes v)$ and $I(\vec y\otimes v)$ is 
\[\sum_{j=1}^m\left|\sum_{i=1}^kb_{ij}(x_i-y_i)\right|\le\sum_{j=1}^m\sum_{i=1}^kb_{ij}|x_i-y_i|=\sum_{i=1}^k|x_i-y_i|\sum_{j=1}^mb_{ij}=\frac{|B|-1}{|B|}\sum_{i=1}^k|x_i-y_i|,\]
hence the map $I(\cdot\otimes v)$ is contracting on each simplex.

Let us show that if $G$ is finitely generated and $d\ge 1$, the we can modify the complex $\Gamma$ in such a way that it becomes connected. Note that if $\limg$ is connected, then the complex constructed above is automatically connected.

Choose for every $U\in\mathcal{V}$ a preimage $U'\in\mathcal{U}$. Fix also a finite generating set $S=S^{-1}\ni 1$ of $G$. Connect now every pair of vertices of $\Gamma$ of the form $U_1'\cdot g, U_2'\cdot sg$ by an edge (if they are different and are not already connected in $\Gamma$) for all $U_1, U_2\in\mathcal{V}$, $g\in G$, and $s\in S$. We get a connected locally finite simplicial complex $\tilde\Gamma$ on which $G$ acts by automorphisms properly and co-compactly. If $d\ge 1$, then dimension of $\tilde\Gamma$ is the same as the dimension of $\Gamma$.

Let $d$ be a metric on $\limg$ satisfying the conditions of Theorem~\ref{th:limgiscontracting}. Then there exists $\epsilon>0$ such that the $d$-distance between any two disjoint elements of the cover $\mathcal{U}$ is greater than $\epsilon$. There exists $n_0$ such that the distance between $U_1'\cdot g\otimes v$ and $U_2'\cdot sg\otimes v$ is less than $\epsilon$ for all $U_1, U_2\in\mathcal{V}$, $g\in G$, $s\in S$, and $v\in\bim^{\otimes n}$, $n\ge n_0$. It follows that the map $I^n:\Gamma\otimes\bim^n\arr\Gamma$ for all $n$ big enough will map the vertices of every edge of $\tilde\Gamma\setminus\Gamma$ to one connected component of $\Gamma$. It follows that for all $n$ big enough the map $I^n:\Gamma\otimes\bim^n\arr\Gamma$ can be extended to a contracting $G$-equivariant map $\tilde\Gamma\otimes\bim^n\arr\Gamma\subset\tilde\Gamma$.
\end{proof}

\subsection{Dynamical asymptotic dimension}

A \emph{groupoid} is a small category of isomorphisms. We consider it as a set of morphisms, where objects of the category are identified with the identity automorphisms (\emph{units} of the groupoid). We denote by $\be(g)$ and $\en(g)$ the \emph{source} and the \emph{range} of an element $g$ of the groupoid, respectively. In other words (taking into account the above identification of objects with units) we have $\be(g)=g^{-1}g$ and $\en(g)=gg^{-1}$. (We compose morphisms from right to left.) A \emph{topological groupoid} is a groupoid with topology on it such that the operations of taking inverse and (partially defined) multiplication are continuous.

For an action of a group $G$ on a topological space $\X$, we can consider the \emph{action groupoid} $G\times\X$ given by the conditions:
\[\be(g, x)=x,\quad\en(g, x)=g(x),\]
and $(g_2, g_1(x))(g_1, x)=(g_2g_1, x)$. We introduce on it the direct product topology. 

The \emph{groupoid of germs} of the action is the quotient of the action groupoid by the equivalence relation identifying $(g_1, x)$ with $(g_2, x)$ whenever there exists a neighborhood $U$ of $x$ such that restrictions of the action of $g_1$ and $g_2$ to $U$ coincide. We introduce the quotient topology on the groupoid of germs.

The following notion was introduced in~\cite{GWY:dad}.

\begin{defi}
\label{def:dad}
Let $\Gr$ be a topological groupoid. Its \emph{dynamical asymptotic dimension} $\dad\Gr$ is the smallest $d$ such that for every compact subset $C\subset\Gr$ there exists an open cover $\{U_0, U_1, \ldots, U_d\}$ of the space of units $\Gr^{(0)}$ such that for every $U_i$ the set $C|_{U_i}=\{g\in C\;:\;\be(g), \en(g)\in U_i\}$ is contained in a compact subgroupoid of $\Gr$.
\end{defi}

\begin{theorem}
\label{th:dad}
Let $(G, \bim)$ be a contracting self-similar group. Denote by $\Gr$ the groupoid of germs of the action of $G$ on the boudary $\xo$ of the tree $T_\bim=\xs$. The following conditions are equivalent.
\begin{enumerate}
\item The topological dimension of the limit space $\lims$ is less than or equal to $d$.
\item For every finite set $A\subset G$ there exists $n\ge 1$ and a partition $V_0, V_1, \ldots, V_d$ of the $n$th level $\alb^n$ of $\xs$ such that for every $i$ the groupoid of isomorphisms between the subtrees $v\xs$ of $\xs$ generated by the maps $g:v\xs\arr g(v\xs)$ for $g\in A$ and $v\in V_i$ such that $g(v)\in V_i$ is finite.
\item Condition (2) for $A$ equal to the nucleus.
\item The dynamical asymptotic dimension of $\Gr$ is less than or equal to $d$.
 \end{enumerate}
\end{theorem}

\begin{proof}
If a cover $\{U_0, U_1, \ldots, U_d\}$ satisfies the conditions of the Definition~\ref{def:dad}, and $\{U_0', U_1', \ldots, U_d'\}$ is another cover such that $U_i'\subset U_i$, then it also satisfies the conditions of the definition. Consequently, in the case when the unit space $\Gr^{(0)}$ is homeomorphic to the Cantor set, we may assume that the sets $U_i$ are pairwise disjoint and clopen. Consequently, if $\Gr^{(0)}=\xo$, then we may assume that there exists $n$ and a partition $\{V_i\}_{i=0}^d$ of the set $\alb^n$ such that $U_i=\bigcup_{v\in V_i}v\xo$. This proves the equivalence of the conditions (4) and (2).

The implication (2)$\Longrightarrow$(3) is obvious.

Let us show the implication (3)$\Longrightarrow$(2). Suppose that (3) holds, and let $A$ be any finite subset of $G$. Let $\{V_i\}_{i=0}^d$ be the partition of $\alb^n$ satisfying the condition (3) for the nucleus $\nuke$. Let $k$ be such that $g|_v\in\nuke$ for all $g\in A$ and $v\in\alb^k$. Let $U_i$ be the set of words $u\in\alb^{n+k}$ such that the suffix of length $k$ of the word $u$ belongs to $V_i$. Then, for every $g\in A$ and $u\in U_i$ such that $g(u)\in U_i$, the restriction $g:u\xo\arr g(u)\xo$ is a map of the form $u_1u_2w\mapsto g(u_1)h(u_2w)$, where $u_1\in\alb^k$, $h\in\nuke$, and $u_2, h(u_2)\in V_i$. It follows that the groupoid generated by such maps is finite, since the groupoid generated by the maps $u_2w\mapsto h(u_2w)$, for $u_2, h(u_2)\in V_i$ and $h\in\nuke$, is finite.

Let us prove the implication (1)$\Longrightarrow$(2).
It follows from Lemma~\ref{lem:adapted} that there exists an open cover $\mathcal{U}$ of the limit space $\lims$ which is liftable to a $G$-invariant cover $\tilde{\mathcal{U}}$ by adapted sets of $\limg$ and such that there exists a partition $\mathcal{U}=\mathcal{U}_0\sqcup\mathcal{U}_1\sqcup\cdots\sqcup\mathcal{U}_d$ such that every set $\mathcal{U}_i$ consists of pairwise disjoint elements.  We will denote by $\tilde{\mathcal{U}}_i$ the set of preimages in $\tilde{\mathcal{U}}$ of elements of $\mathcal{U}_i$. Then the elements of $\tilde{\mathcal{U}}_i$ are also pairwise disjoint.

Let $A\subset G$ be a finite set such that $1\in A$. Pick $\xi\in\limg$.
It follows from contraction and Lebesgue lemma that there exists $n_A$ such that for every $v\in\alb^n$ there exists $U_v\in\tilde{\mathcal{U}}$ such that $\xi\cdot g\otimes v\in U$ for all $g\in A$. Let $i_v\in\{0, 1, \ldots, d\}$ be such that $U_v\in\tilde{\mathcal{U}}_{i_v}$. Note that set of chosen sets $U_v$ is finite.
We get a partition of $\alb^{n_A}$ into $d+1$ sets $C_i=\{v\in\alb^{n_A}\;:\; i_v=i\}$.

For every $i\in\{0, 1, \ldots, d\}$, consider the set $\mathcal{H}_i$ of isomorphisms $v_1\xs\arr v_2\xs$ of the form $v_1w\mapsto v_2h(w)$, where $v_1, v_2\in\alb^n$ and $h\in G$ are such that both points $\xi\otimes v_1, \xi\otimes v_2\cdot h$ belong to a single element $U\in\tilde{\mathcal{U}}_i$. Note that this implies that $h$ belongs to the stabilizer of $U$ (since $\tilde U$ is adapted) and hence $\xi\otimes v_2\in U$. It follows that $\mathcal{H}_i$ is a finite category of isomorphisms between subtrees of $\xs$. 

Suppose that $v\in\alb^n$ belongs to $C_i$ and let $g\in A$. Then $\xi\otimes v, \xi\otimes g\cdot v\in U_v\in\tilde{\mathcal{U}}_i$. The isomorphism $v\xs\arr g(v\xs)$ acts by $vw\mapsto g(v) g|_v(w)$, where $v, g(v)\in\alb^{n_A}$ and $g|_v$ are such that $\xi\otimes v$ and $\xi\otimes g(v)\cdot g|_v$ belong to $U_v$. It follows that this isomorphism belongs to $\mathcal{H}_i$.

It remains to prove the implication (3)$\Longrightarrow$(1). Let $n$ be such that condition (2) for $A$ equal to the nucleus is true. Consider the corresponding partition of the $n$th level of the tree $\xs$ into sets $V_0, V_1, \ldots, V_d$. Let $m>n$, and consider the induced partition $V_{m, 0}, V_{m, 1}, \ldots, V_{m, d}$ of $\alb^m$, where $v\in V_{m, i}$ if and only if the length $n$ prefix of $v$ belongs to $V_i$. 

Consider the partition of $\limg$ into the tiles of $m$th level, and partition them into $d+1$ sets $T_0, T_1, \ldots, T_d$, so that $\til\otimes v\in T_i$ and only if $v\in V_{m, i}$. For every $i\in\{0, 1, \ldots, d\}$ consider the graph with the set of tiles $T_i$ in which two tiles are connected if and only if they intersect. Then the connected components of this graph will be the orbits of the groupoid from condition (2). It follows that the components are of uniformly bounded size (with an estimate not depending on $m$). Let $\mathcal{U}_i$ be the set of unions of tiles corresponding to the connected components of this graph. Then the elements of $\mathcal{U}_i$ are pairwise disjoint, and each element is a union of a uniformly bounded number of tiles of the $m$th level. Each of the sets $\mathcal{U}_i$ is $G$-invariant, and we get a closed cover of multiplicity $\le d+1$ of the limit space $\lims$ into unions of $m$th level tiles such that every element of the cover is a union of a uniformly bounded number of tiles. It follows that the diameters of the elements of these covers uniformly converge to zero when $m\to\infty$, hence the topological dimension of $\lims$ is at most $d$.
\end{proof}

\subsection{Asymptotic dimension of the orbital graphs}

Consider the graph $\Gamma_w$ with the set of vertices equal to the orbit of $w$, in which two vertices are connected by an edge if and only if they are of the form $g(w), hg(w)$, for $g\in G$ and $h\in\nuke$. The graph $\Gamma_w$ is the \emph{orbital graph} of $G$ acting on $\xo$.

\begin{theorem}
\label{th:asdym}
Let $(G, \bim)$ be a contracting finitely generating self-replicating group. Then the topological dimension of its limit space is equal to the asymptotic dimension of every orbital graph of the action of $G$ on $\xo$.
\end{theorem}

\begin{proof}
Let us choose a basis $\alb\subset\bim$, and let $\nuke$ be the corresponding nucleus. Then $\nuke$ is a generating set of $G$ (see~\cite[Proposition~2.11.3]{nek:book}), and we can consider the orbital graphs with respect to this generating set.

Let us show that the asymptotic dimension of the orbital graphs is not greater than $\dim\lims$. Let $V_0, V_1, \ldots, V_d$ be the partition of the $n$th level of the tree satisfying condition (3) of Theorem~\ref{th:dad}.

Choose $m\ge 0$. Define $V_{m, i}$ as the set of words $x_1x_2\ldots \in\xo$ such that $x_{m+1}x_{m+2}\ldots x_{m+n}\in V_i$. If $x_1x_2\ldots, y_1y_2\ldots$ belong to the same connected component of the graph induced in $\Gamma_w$ by $V_{m, i}$, then $x_{m+1}x_{m+2}\ldots x_{m+n}$ and $y_{m+1}y_{m+2}\ldots y_{m+n}$ belong to the same connected component of the graph induced by $V_i$ in the graph of the action of $G$ on $\alb^n$. The size of the latter components are uniformly bounded. It follows that the size of the components of $\Gamma_w$ are also uniformly bounded. We get, therefore, a cover of $\Gamma_w$ by sets of uniformly bounded diameter. 

It remains to show that as $m$ increases to infinity, the distances between different components of $V_{m, i}$ grows to infinity. Suppose that it is not true. Then there exists sequences $g_m\in G$, $x_{1, m}x_{2, m}\ldots\in\xo$ such that $g_m$ has bounded length with respect to the nucleus, 
and the words $x_{m+1, m}x_{m+2, m}\ldots x_{m+n, m}$, $y_{m+1, m}y_{m+2, m}\ldots y_{m+n, m}$ belong to different components of $V_i$ for every $m$, where $y_{1, m}y_{2, m}\ldots=g_m(x_{1, m}x_{2, m}\ldots)$.
We have $y_{m+1, m}y_{m+2, m}\ldots y_{m+n, m}=g_m|_{x_{1, m}x_{2, m}\ldots x_{m, m}}(x_{m+1, m}x_{m+2, m}\ldots x_{m+n, m})$. But $g_m|_{x_{1, m}x_{2, m}\ldots x_{m, m}}$ belongs to the nucleus for all $m$ big enough, which is a contradiction.

Let us show that $\dim\lims$ is less than or equal to the asymptotic dimension of $\Gamma_w$.

Let $w\in\xo$, and consider the set of bi-infinite words $\ldots x_{-2}x_{-1}\;.\;g(w)$ for $g\in G$. Here the dot will mark the place between coordinates number 0 and -1 in the sequence. We consider it as a direct product of the space $\xmo$ with the discrete $G$-orbit of $w$. Consider the asymptotic equivalence relation: two sequences $(x_n)_{n\in\Z}$ and $(y_n)_{n\in\Z}$ are equivalent if and only if there exists a sequence $g_n$ taking values in a finite subset of $G$ (equivalently, in the nucleus) such that $g_n(x_{-n}x_{-n+1}\ldots)=y_{-n}y_{-n+1}\ldots$. Let $\mathcal{L}_w$ be the quotient space.

The topological dimension of $\mathcal{L}_w$ is equal to the topological dimension of $\lims$, by the same arguments as in the proof of Proposition~\ref{prop:dimlimglims}.

Note that the equivalence relation is invariant under the shift \[\si:\ldots x_{-2}x_{-1}\;.\;x_0x_1\ldots\mapsto \ldots x_{-1}x_{0}\;.\;x_{1}x_2\ldots.\] Since $(G, \bim)$ is assumed to be self-replicating, sequences that differ from each other only by the first letter belong to the same $G$-orbit. It follows that the shift induces a homeomorphism $\mathcal{L}_{x_0x_1\ldots}\arr\mathcal{L}_{x_1x_2\ldots}$, which we will also denote by $\si$.

The space $\mathcal{L}_w$ is a union of the sets $\mathcal{T}_{g(w)}$ equal to the images of the sets $\xmo\times\{g(w)\}$. Applying the shift to the sets $\mathcal{T}_{g(w)}$ we get tiles of the \emph{$n$th level} equal to the images of the sets of sequences $(x_k)_{k\in\Z}$ with specified values of $x_k$ for all $k\ge -n$. (The tiles become smaller when we increase $n$.) Note that here the level of a tile can be negative.

The shift $\si$ maps the $n$th level tiles to $(n+1)$st level tiles.
Two tiles $\til_{w_1}, \til_{w_2}$ intersect if and only if there exists an element $g$ of the nucleus $\nuke$ such that $g(w_1)=w_2$ (see~\cite[Section~3.3.2]{nek:book}). It follows that the adjacency of the tiles of the level number 0 is described by the orbital graph $\Gamma_w$. The adjacency of the tiles of the $n$th level is described by the graph $\Gamma_{\sigma^n(w)}$. Here we allow $n$ to be negative, by taking any preimage of $w$ under $\sigma^{-n}$ (as they all belong to the same $G$-orbit).

For two points $\xi_1, \xi_2\in\mathcal{L}_w$, define $\ell(\xi_1, \xi_2)$ as the maximal $n$ such that $\xi_i$ belong to intersecting $n$th level tiles. 

We have then $\ell(\sigma^k(\xi_1), \sigma^k(\xi_2))=\ell(\xi_1, \xi_2)+k$.

One can show, see~\cite{nek:hyperbolic}, that for all positive $\alpha$ that are small enough, there exists a metric $d_\alpha$ on $\mathcal{L}_w$ such that $C^{-1}e^{-\alpha\ell(\xi_1, \xi_2)}\le d_\alpha(\xi_1, \xi_2)\le Ce^{-\alpha\ell(\xi_1, \xi_2)}$ for all $\xi_1, \xi_2\in\mathcal{L}_w$.

The shift $\sigma$ is almost a homothety in the following sense. There exists $C$ such that for every $\xi_1, \xi_2\in\mathcal{L}_w$ and every $n\in\Z$ we have
\[C^{-1}e^{-\alpha n}d_\alpha(\xi_1, \xi_2)\le d_\alpha(\sigma^n(\xi_1), \sigma^n(\xi_2))\le Ce^{-\alpha n}d_\alpha(\xi_1, \xi_2).\]

For a given tile $\til_w$ (of the level $0$), denote by $\til'_w$ the union of all tiles of the level 0 intersecting it. Then $\til'_w$ is a compact neighborhood of $\til_w$ (see~\cite[Proposition~3.4.1]{nek:book}). The natural quotient map $\mathcal{L}_w\arr\lims$ is an open countable-to-one map, and the interior of $\til'_w$ is mapped surjectively onto $\lims$. It follows that the topological dimension of the interior of $\til'_w$ is equal to $\dim\lims$. Consequently, $\dim\til'_w\ge\dim\lims$. On the other hand, we have $\dim\mathcal{L}_w=\dim\lims$, and $\til'_w$ is a subset of $\mathcal{L}_w$, so we have $\dim\til'_w\le\dim\lims$. Consequently, $\dim\til'_w=\dim\til_w$.

Let $A_1, A_2$ be subsets of the $G$-orbit of $w\in\xo$, and denote by $\til_{A_i}$ the union of all tiles $\til_u$ for $u\in A_i$. It follows from the definition of the topology on $\mathcal{L}_w$ that if a bijection $\phi:A_1\arr A_2$ is an isomorphism of the subgraphs of $\Gamma_w$ spanned by $A_i$, then the map $\ldots x_{-2}x_{-1}\;.\;u\mapsto\ldots x_{-2}x_{-1}\;.\;\phi(u)$ induces a homeomorphism $\til_{A_1}\arr\til_{A_2}$.

The set $\til'_u$ is equal to $\til_A$, where $A$ is the ball in $\Gamma_w$ of radius $1$ with center in $u$. There exists an open subset $U\subset\xo$ such that the balls in $\Gamma_w$ of radius $1$ with center in each $u\in U$ are isomorphic to each other. It follows that for every $u\in U$ a homeomorphic copy of $\til'_u$ appears in each leaf $\mathcal{L}_w$.

Suppose that the asymptotic dimension of $\Gamma_w$ is less than or equal to $d$.
Since $\Gamma_w$ describes the adjacency of tiles of the $0$th level of $\mathcal{L}_{w}$, there exists $C$ and a cover $\mathcal{U}=\mathcal{U}_0\cup\ldots\cup\mathcal{U}_d$ of $\mathcal{L}_w$ by sets of diameter $\le C_1$ (unions of the tiles corresponding to the respective elements of the cover of $\Gamma_w$) such that elements of each $\mathcal{U}_i$ are pairwise disjoint.

Applying $\sigma$ to it, we will get a cover of $\Gamma_{\sigma^n(w)}$ by sets of size at most $CC_1e^{-n\alpha}$ with the same condition. Finding a homeomorphic copy of $\til_w'$ in $\Gamma_{\sigma^n(w)}$, we prove, using the Lebesgue covering lemma, that $\dim \til_w'$ is at most $d$.
\end{proof}

\section{One-dimensional case}

\subsection{Properties of contracting groups of dimension $\le 1$}
\begin{theorem}
\label{th:zerodimdescr}
Let $(G, \bim)$ be a contracting self-similar group. The following conditions are equivalent.
\begin{enumerate}
\item The limit space of $(G, \bim)$ is zero-dimensional.
\item The group $G$ is locally finite.
\item The subgroup of $G$ generated by the nucleus is finite.
\end{enumerate}
\end{theorem}

\begin{proof}
Equivalence of (1) and (2) follows from condition (2) of Theorem~\ref{th:dad}. Similarly, equivalence of (1) and (3) follows from condition (3) of the same theorem.
\end{proof}

\begin{example}
The group of \emph{finitary automorphisms} of $\xs$ is the group of automorphism $g$ for which there exists $n$ such that $g|_v=1$ for all $v\in\alb^n$. It is contracting with the nucleus equal to $\{1\}$. The limit space $\lims$ is just $\alb^{-\omega}$, since the asymptotic equivalence is trivial. In general, for every contracting group $G$, the group generated by $G$ and the group of finitary automorphisms is contracting with the same nucleus as the nucleus of $G$.
\end{example}

\begin{example}
Every finite self-similar group is contracting. For example, consider the group acting on $\{0, 1, 2\}^\omega$ and generated by $a=(01)(a, a, 1)$, $b=(02)(b, 1, b)$, and $c=(12)(1, c, c)$. It is the symmetric group of permutations of $\{0, 1, 2\}$ acting on each coordinate of elements of $\{0, 1, 2\}^\omega$ (i.e., diagonally). 

The limit orbi-space of this group is the quotient of $\{0, 1, 2\}^{-\omega}$ by similarly defined action of the symmetric group on it. It is also homeomorphic to the Cantor set.
\end{example}

\begin{theorem}
\label{th:onedimdescription}
Let $(G, \bim)$ be a contracting faithful self-similar finitely generated group.
The limit space of $(G, \bim)$ has topological dimension 1 if and only if the self-similar group $(G, \bim)$ is the faithful quotient of a contracting virtually free group.
\end{theorem}

\begin{proof}
For every $m\ge 1$, a self-similar group $(G, \bim)$ is contracting if and only if $(G, \bim^{\otimes m})$ is contracting. Consequently, it is enough to prove that $(G, \bim^{\otimes m})$ is the faithful quotient of a contracting virtually free group for some $m\ge 1$.

By Theorem~\ref{th:contractingmodeldim}, there exists a connected graph $\Gamma$, a proper co-compact action of $G$ on it by automorphisms, and a contracting $G$-equivariant map $I:\Gamma\otimes\bim^{\otimes m}\arr\Gamma$. In order to simplify the notation, we replace $\bim^{\otimes m}$ by $\bim$, i.e., assume that $m=1$. 

Our virtually free group $\tilde G$ such that $G$ is its faithful quotient will be the fundamental group of the graph of groups $\Gamma/G$.

Let us describe it directly. Choose a basepoint $t\in\Gamma$ (e.g., a vertex). The elements of $\tilde G$ are pairs $(g, [\gamma])$, where $\gamma$ is a homotopy class of a path in $\Gamma$ starting in $t$ and ending in $t\cdot g$. They are multiplied  by the rule
\[(g_1, [\gamma_1])\cdot (g_2, [\gamma_2])=(g_1g_2, [\gamma_1\cdot g_2][\gamma_2]),\]
where $\gamma_1\cdot g_2$ is the path obtained from $\gamma_1$ by the action of $g_2$. Note that $\gamma_2$ ends in $t\cdot g_2$, which is the beginning of $\gamma_1\cdot g_2$. The path $\gamma_1$ ends in $t\cdot g_1$, hence the path $\gamma_1\cdot g_2$ ends in $t\cdot g_1g_2$.

It is checked directly that multiplication is associative, and that the element $(1, [t])$, where $[t]$ represents the trivial path at $t$, is a neutral element with respect to multiplication. We also have $(g, [\gamma])^{-1}=(g^{-1}, [\gamma\cdot g^{-1}]^{-1})$, hence $\tilde G$ is a group.

We have a natural epimorphism $(g, [\gamma])\mapsto g$ from $\tilde G$ to $G$.

Let us define the corresponding biset $\tilde{\bim}$. Its elements will be pairs $(x, [\delta])$, where $x\in\bim$, and $\delta$ is a path starting in $t$ and ending in $I(t\otimes x)$. The right action is defined by the rules
\[(x, [\delta])\cdot(g, [\gamma])=(x\cdot g, [\delta\cdot g][\gamma]).\]
Note that $\gamma$ ends in $t\cdot g$, which is the beginning of $\delta\cdot g$. The end of $\delta\cdot g$ is $I(t\otimes x)\cdot g=I(t\otimes x\cdot g)$.

The kernel $K$ of the epimorphism $G\arr\tilde G$ consists of the elements of the form $(t, [\gamma])$, where $\gamma$ is a loop in $\Gamma$. It follows that two elements $(x_i, [\delta_i])$, for $i=1, 2$, belong to the same right $K$-orbit if and only if $x_1=x_2$.

The left action is given by
\[(g, [\gamma])(x, [\delta])=(g\cdot x, [I(\gamma\otimes x)][\delta]).\]
Note that $I(\gamma\otimes x)$ begins in $I(t\otimes x)$ and ends in $I(t\cdot g\otimes x)=I(t\otimes g\cdot x)$.

It follows from the description of the right $K$-orbits and the definition of the biset structure on $\tilde{\bim}$ that $(G, \bim)$ is the faithful quotient of $(\tilde G, \tilde{\bim})$.

Let $\tilde X$ be the universal covering of $\Gamma$. We can represent the points of $\tilde X$ as the homotopy classes of paths $\delta$ in $\Gamma$ starting in $t$. The group $\tilde G$ acts naturally on $\tilde X$ in this interpretation:
\[[\delta]\cdot (g, [\gamma])=[\delta\cdot g][\gamma].\]
Note that the covering map $\tilde\Gamma\arr\Gamma$ and the natural epimorphism $G\arr\tilde G$ project this action of $\tilde G$ to the action of $G$ on $\Gamma$. It follows that the action of $\tilde G$ on $\tilde\Gamma$ is co-compact.

Since $\tilde\Gamma$ is a locally finite graph and the action of $\tilde G$ on it is an action by automorphisms (since the action of $G$ on $\Gamma$ is an action by automorphisms of the graph), it is enough, in order to prove properness of the action, to prove that the $\tilde G$-stabilizers of points of $\tilde\Gamma$ are finite.

If $[\delta]\cdot (g, [\gamma])=[\delta]$, then $[\delta]$ is homotopic to $[\delta\cdot g][\gamma]$. In particular, the end of $\delta$ is fixed by $g$. Since the action of $G$ on $\Gamma$ is proper, we get only finitely many choices for $g$. For each choice of $g$, the homotopy class $[\gamma]$ is equal to $[\delta\cdot g]^{-1}[\delta]$, i.e., is unique. It follows that the $\tilde G$-stabilizers of points of $\tilde\Gamma$ are isomorphic to the $G$-stabilizers of the corresponding points of $\Gamma$. In particular, the action $\tilde G$ on $\tilde\Gamma$ is proper.

Let us define the map $\tilde I:\tilde\Gamma\otimes\tilde{\bim}\arr\tilde\Gamma$ by
\[\tilde I([\delta]\otimes(x, [\beta]))=[I(\delta\otimes x)][\beta].\]
Recall that, by the definition of $\tilde\bim$, $\beta$ is a path starting in $t$ and ending in $I(t\otimes x)$. Consequently, $I(\delta\otimes x)$ will start in the end of $\beta$, so that the concatenation is defined. Note also that the end of $I(\delta\otimes x)$ is equal to $I(v\otimes x)$, where $v$ is the end of $\delta$. In particular, the covering map $\tilde\Gamma\arr\Gamma$ and the natural map $\tilde{\bim}\arr\bim$ transform $\tilde I$ into $I$. Since $I(\cdot\otimes x)$ contracts the lengths of the edges of $\Gamma$, the map $\tilde I(\cdot\otimes (x, [\beta]))$ contracts the lengths of edges of $\tilde\Gamma$.

We get a contracting equivariant map $\tilde I:\tilde\Gamma\otimes\tilde{\bim}\arr\tilde\Gamma$. Consequently, the group $(\tilde G, \tilde{\bim})$ is contracting. Since $\tilde G$ acts properly on the tree $\tilde\Gamma$, it is virtually free.

Conversely, suppose that there exists a virtually free contracting self-similar group $(\tilde G, \tilde{\bim})$ such that its faithful quotient is $(G, \bim)$.

Then, by Theorem~\cite[Theorem~6.7]{nek:models} there exists $m$ such that the natural equivariant map $g\otimes v\mapsto g|_v$ from the Rips complex $\Delta_{G, S}$ is homotopic to a contracting map $I:\Delta_{G, S}\otimes\bim^{\otimes m}\arr\Delta_{G, S}$ for some $m$, where $S$ is an arbitrary sufficiently large symmetric generating set of $G$.

By~\cite{karrasspietrowskisolitar}, there exists a proper co-compact action of $G$ on a tree $\Gamma$. Let us show that for all sufficiently large $S$ there exist $G$-equivariant Lipschitz maps $\Gamma\arr\Delta_{G, S}$ and $\Delta_{G, S}\arr\Gamma$.

Let us construct the first map. Choose a spanning tree of the finite graph $\Gamma/G$, and let $T$ be its lift to $\Gamma$. Then $\Gamma$ is covered by the sets $T\cdot g$. We may assume that $\Gamma$ has a free orbit (by adding an extra vertex of degree 1 to the graph of groups $\Gamma/G$). 

Since the action of $G$ on $\Gamma$ is proper, we may assume that the set of elements $g\in G$ such that $T\cap T\cdot g\ne\emptyset$ is contained in $S$. Let $v$ be a vertex or an edge of $\Gamma$, and let $A_v$ be the set of elements of $g$ such that $v\in T\cdot g$. Then $A_v$ is a simplex in $\Delta_{G, S}$. We assume that for every edge $e$ of $\Gamma$ and every $g\in G$ such that $e\cdot g=e$, the endpoints of $e$ are fixed by $g$. (Otherwise, we pass to the barycentric subdivision of $\Gamma$.) Then for every edge $e$ and an endpoint of $v$ of $e$, we have $A_e\subset A_v$. It follows that if we map every vertex $v$ of $\Gamma$ to the barycenter of $A_v$, and the midpoint of every edge $e$ of $\Gamma$ to the barycenter of $A_e$, then for every edge $e$ and its endpoint $v$, the image of the midpoint of $e$ and the image of $v$ form a 1-simplex in the barycentric subdivision of $\Delta_{G, S}$. Consequently, we get a simplicial map $\phi$ from $\Gamma$ to the barycentric subdivision of $\Delta_{G, S}$. The map is obviously $G$-equivariant.

Let us contract an equivariant map $\psi:\Delta_{G, S}\arr\Gamma$.
It follows from the definition of the map $\phi$ that a fiber of $\phi$ is contained in the intersections of a set of trees $T\cdot g$. Since every such an intersection is a tree, and $\phi$ is simplicial, the fibers of $\phi$ are trees of uniformly bounded size. After contracting them in $\Gamma$, we may assume that $\phi:\Gamma\arr\Delta_{G, S}$ is an embedding.

Since there exists a free orbit on $\Gamma$, the image $\phi(\Gamma)$ contains the set of vertices of $\Delta_{G, S}$.
We can construct a $G$-equivariant map $\psi:\Delta_{G, S}\arr\Gamma$ by mapping every edge of $\Delta_{G, S}$ to the unique arc in $\phi(\Gamma)$ connecting its endpoints. This can be done by a Lipschitz $G$-equivariant map. Since $\phi(\Gamma)$ is a tree hence contractible, we can extend this map from the 1-skeleton to the 2-skeleton, then to 3-skeleton, and eventually to all of $\Delta_{G, S}$. Since the action of $G$ on $\Delta_{G, S}$ is proper and co-compact, we can do it in a $G$-equivariant Lipschitz way.

We know that there exists a contracting $G$-equivariant map $I:\Delta_{G, S}\otimes\bim^{\otimes m}\arr\Delta_{G, S}$. Let $\phi:\Gamma\arr\Delta_{G, S}$ and $\psi:\Delta_{G, S}\arr\Gamma$ be $G$-equivariant Lipschitz maps. Suppose that they are $\Lambda$-Lipschitz, and the map $I_v:t\mapsto I(t\otimes v)$ is $\lambda$-contracting for some $\lambda\in (0, 1)$ and all $t\in\Delta_{G, S}$ and $v\in\bim^{\otimes m}$. Then the map $t\mapsto \psi(I^n(\phi(t)\otimes v))$ for $v\in\bim^{\otimes mn}$ maps points of $\Gamma$ on distance $l$ from each other to points of $\Gamma$ on distance at most $\Lambda^2\lambda^n$ from each other. By taking $n$ large enough, we will find a contracting $G$-equivariant map $\Gamma\otimes\bim^{\otimes mn}\arr\Gamma$. It will imply, by Theorem~\ref{th:contractingmodeldim}, that $\limg$ has topological dimension at most $1$. Since the virtually free group $G$ is infinite, the dimension is equal to 1, see Theorem~\ref{th:zerodimdescr}.
\end{proof}

\begin{example}
Consider the \emph{Basilica group} generated by the wreath recursion
\[a=\sigma(1, b),\qquad b=(1, a).\]

Let us show that it is contracting without using any relation, i.e., that the free group with this wreath recursion is contracting. 

We have $ab^{-1}=\sigma(1, ba^{-1}), ba^{-1}=\sigma(ab^{-1}, 1)$, so $ab^{-1}$ and $ba^{-1}$ must belong to the nucleus. Let us show that $\nuke=\{1, a, a^{-1}, b, b^{-1}, ab^{-1}, ba^{-1}\}$ is the nucleus. It is enough to show that the sections of the elements $\{a, a^{-1}, b, b^{-1}, ab^{-1}, ba^{-1}\}\cdot\{a, a^{-1}, b, b^{-1}\}$ in sufficiently long words belong to $\nuke$. It is enough to consider $a^2, ab, b^2, ba, ab^{-2}, a^2b^{-1}, aba^{-1}, bab^{-1}$. (We do not have to consider separately their inverses, as $\nuke^{-1}=\nuke$.) We have $a^2=(b, b)$, $b^2=(1, a^2)$, $ab=\sigma(1, ba)$, $ba=\sigma(a, b)$, $ab^{-2}=\sigma(1, ba^{-2})$, $a^2b^{-1}=(b, ba)$, $aba^{-1}=(b^{-1}ab, 1)$, $bab^{-1}=\sigma(a, ba^{-1})$. We see that sections of all elements eventually belong to $\nuke$. 

We have not used any relations between the generators $a, b$, hence the recursion is contracting on the free group. Consequently, the limit space of the Basilica group has topological dimension 1.
\end{example}

\begin{example}
Consider the Gupta-Sidki group~\cite{gupta-sidkigroup} acting on $\{0, 1, 2\}^*$ and generated by the cyclic permutation $a=(012)$ of the first level (with all sections trivial) and the element
\[b=(a, a^{-1}, b).\]

We see that this wreath recursion is not contracting on the free group, since we have $b^n=(a^n, a^{-n}, b^n)$. However, it is easy to see that $a^3=1$ and hence $b^3=1$ in the faithful quotient. We can consider, therefore, the wreath recursion on the free product $(\Z/3\Z)*(\Z/3\Z)$, presented by the defining relations $a^3=b^3=1$. Since the sections of $ab, a^{-1}b, ab^{-1}$ in one-letter words belong to the set $\{1, a, a^{-1}, b, b^{-1}\}$, we see that the wreath recursion is contracting on the free product. Consequently, the limit space of the Gupta-Sidki group is one-dimensional.
\end{example}

\subsection{Finitely presented cases}

\begin{theorem}
\label{th:onedimfinpres}
Let $(G, \bim)$ be a finitely presented faithful contracting self-replicating group such that the limit space $\limg$ has topological dimension 1. 
Then $(G, \bim)$ is equivalent to one of the following groups:
\begin{enumerate}
\item $\img{z^n}$ for some $n\in\Z$, $|n|\ge 2$.
\item $\img{T_n}$, where $T_n$ is a Chebyshev polynomial of degree $n$.
\item $\img{-T_n}$, where $T_n$ is a Chebyshev polynomial of an odd degree $n$.
\end{enumerate} 
\end{theorem}

The iterated monodromy group of $-T_d$ for even $d$ is equivalent to the iterated monodromy group of $T_d$ (since the polynomials are conjugate), so it is not listed in Theorem~\ref{th:onedimfinpres}.

\begin{proof}
By Theorem~\ref{th:onedimdescription}, we may assume, after replacing $\bim$ by an iterate $\bim^{\otimes n}$, that $(G, \bim)$ is the iterated monodromy group of a correspondence $f, \iota:\Gamma_1\arr\Gamma$, where $\Gamma$ is a finite connected graph of groups, $f$ is a covering of graphs of groups, and a morphism $\iota:\Gamma_1\arr\Gamma$ is contracting on the underlying graphs. The fundamental group $\tilde G$ of $\Gamma$ is virtually free, the  associated biset $(\tilde G, \tilde\bim)$ over the fundamental group is contracting, and $(G, \bim)$ is the faithful quotient of $(\tilde G, \tilde\bim)$.

Choose a basis $\alb$ of $\tilde\bim$. It follows from the definition of the associated wreath recursion that if $g\in\tilde G$ is an element of the nucleus of $(\tilde G, \tilde\bim)$ such that its image in $G$ is trivial, then $g$ is conjugate to an element of a vertex group of $\Gamma$. But if we start with a $G$-adapted cover of the limit space of $\limg$ by sufficiently small sets, the isotropy groups of the vertices of the corresponding complex $\Delta_{G, S}$, used in the proof of Theorem~\ref{th:onedimdescription} will be sugroups of $G$, i.e., will be faithfully represented in it. It follows, we may assume that the nucleus of $(\tilde G, \tilde\bim)$ does not contain non-trivial elements of $\tilde G$ that are mapped to trivial elements of $G$.

Suppose that $G$ is finitely presented. Then it follows from~\cite[Proposition~2.13.2]{nek:book} that the kernel of the epimorphism $\tilde G\arr G$ is contained in the kernel of some iteration $\phi^n:\tilde G\arr S_{d^n}\ltimes {\tilde G}^{d^n}$ of the wreath recursion $\phi$. It follows that $G$ is isomorphic to a subgroup of $S_{d^n}\ltimes {\tilde G}^{d^n}$. Since the group is self-replicating, the homomorphism from the stabilizer of any word of length $n$ onto the corresponding factor of ${\tilde G}^{d^n}$ (i.e., the associated virtual endomorphism) is surjective. It follows that a finite index subgroup of $G$ is  embedded into the direct product ${\tilde G}^{d^n}$ so that projection of the image onto each direct factor has finite index in $\tilde G$. But faithful contracting groups have no non-abelian free subgroups, see~\cite{nek:free}. Consequently, both $\wt G$ and $G$ are virtually abelian.

By~\cite[Theorem~6.1.6]{nek:book}, if $(G, \bim)$ is a self-replicating virtually abelian group then there exists a homeomorphism of the limit space $\limg$ with the group $\R^m$ conjugating the action of $G$ on $\limg$ to a proper co-compact action of $G$ on $\R^m$ by affine transformations. It follows that in the one-dimensional case the limit space is $\R$ and $G$ acts on it properly and co-compactly by affine transformations $x\mapsto ax+b$. It follows from properness of the action that $a=\pm 1$. It follows then from properness and cocompactness that the action of $G$ on the limit space is either conjugate to the action of $\Z$ on $\R$ or to the action of the dihedral group $D_\infty$ generated by $x\mapsto -x$ and $x\mapsto 1-x$.

The limit dynamical system is induced then by an expanding automorphism $x\mapsto dx$ of $\R$. We have then four cases: either $G$ is $\Z$ or $D_\infty$, and either $d$ is positive or negative. These cases lead to the examples of the iterated monodromy groups listed in the theorem.
\end{proof}

The wreath recursions for the self-similar groups mentioned in Theorem~\ref{th:onedimfinpres} are as follows. We use the alphabet $\alb=\{0, 1, 2, \ldots, d-1\}$ for $d\ge 2$.

The iterated monodromy group of $z^d$ is generated by one element $a$ satisfying the wreath recursion
\[a=\sigma(1, 1, \ldots, 1, a),\]
where $\sigma$ is the cyclic permutation $(0,1,2,\ldots, d-1)$.

The iterated monodromy group of $z^{-d}$ is generated by 
\[a=\sigma(1, 1, \ldots, 1, a^{-1}).\]

The iterated monodromy group of $T_d$ for even $d$ is generated by 
\[a=\alpha,\qquad b=\beta(a, 1, 1, \ldots, 1, 1, b),\]
where permutations $\alpha, \beta$ are
\[\alpha: 0\leftrightarrow 1, 2\leftrightarrow 3, \ldots (d-2)\leftrightarrow (d-1)\]
and
\[\beta: 1\leftrightarrow 2, 3\leftrightarrow 4, \ldots, (d-3)\leftrightarrow (d-2).\]

The iterated monodromy group of $T_d$ for odd $d$ is generated by 
\[a=\alpha(a, 1, 1, \ldots, 1),\qquad b=\beta(1, 1, \ldots, 1, b),\]
where
\[\alpha: 1\leftrightarrow 2, 3\leftrightarrow 4, \ldots, (d-2)\leftrightarrow (d-1)\]
and
\[\beta: 0\leftrightarrow 1, 2\leftrightarrow 3, \ldots, (d-3)\leftrightarrow (d-2).\]

The iterated monodromy group of $-T_d$ for odd $d$ is generated by 
\[a=\alpha(b, 1, 1, \ldots, 1),\qquad b=\beta(1, 1, \ldots, 1, a)\]
for the same permutations $\alpha$ and $\beta$.

\section{Examples}

\subsection{Groups generated by automata of polynomial activity growth}

Let $g$ be an automorphism of the tree $\xs$. Consider the \emph{activity growth function}
\[\alpha_g(n)=|\{v\in\alb^n\;:\;g|_v\ne 1\}|.\]

We say that $g$ is \emph{finite-state} if the set $\{g|_v\;:\;v\in\xs\}$ of all its sections is finite. The set of finite-state automorphisms $g$ of $\xs$ such that  $\alpha_g(n)$ is bounded by a polynomial of degree $d$ is a subgroup of the automorphism group of the tree $\xs$, see~\cite{sid:cycl}. We will denote it $\mathcal{P}_d(\alb)$.

A \emph{non-trivial cycle} (of the Moore diagram) of $g$ is a set of the form $\{g|_v=g|_{vx_1x_2\ldots x_k}, g|_{vx_1}, g|_{vx_1x_2}, \ldots, g|_{vx_1x_2\ldots x_{k-1}}\}$ of non-trivial sections of $g$.

An automorphism $g$ of $\xs$ belongs to $\mathcal{P}_d(\alb)$ if and only if the cycles in the Moore diagram of $g$ passing through non-trivial elements are disjoint and the maximal length of a sequence of cycles $C_1, C_2, \ldots, C_k$
such that elements of $C_i$ are sections of the elements of $C_{i-1}$ for every $i$ is $d+1$.

Suppose that $G$ is a finitely generated self-similar subgroup of $\mathcal{P}_d(\alb)$. Let $S$ be a finite generating set of $G$ such that $g|_x\in S$ for all $g\in S$ and $x\in\alb$, which exists since all elements of $\mathcal{P}_d(\alb)$ are assumed to be finite-state. Let $N$ be a number divisible by the lengths of all non-trivial cycles of all elements of $S$. Let $C$ be the set of words $v\in\alb^N$ such that $g|_v=g$ for some $g\in S$, and let $\mathcal{C}$ be the set of subtrees $v\xs$ for $v\in C$. For every $v\in C$ and $g\in S$ such that $g|_v=g$, consider the isomorphism $v\xs\arr g(v)\xs$ induced by $g$ (i.e., the map $vw\mapsto g(vw)=g(v)g|_v(w)$). Let $\mathcal{G}$ be the groupoid generated by these isomorphisms. The following theorem is proved in~\cite{nekpilgrimthurston}.

\begin{theorem}
\label{th:pold}
The group $G$ is contracting if and only if the groupoid $\mathcal{G}$ is finite. Moreover, the nucleus is equal to the set of sections $g|_v$ for the elements $g:v\xs\arr g(v)\xs$ of $\mathcal{G}$.

A contracting self-similar group $G$ is a subgroup of $\mathcal{P}_d(\alb)$ for some $d$ if and only if the boundary of the tile $\til\subset\limg$ is countable.
\end{theorem}

The smallest $d$ for which $G$ is contained in $\mathcal{P}_d(\alb)$ is equal to the Cantor-Bendixson rank of $\partial\til$ minus one. For example, $G\le\mathcal{P}_0(\alb)$ if and only if $\partial\til$ is finite. The group is contained in $\mathcal{P}_1(\alb)$ if and only if $\partial\til$ has a finite number of non-isolated points. All self-similar subgroups of $\mathcal{P}_0(\alb)$ are contracting (follows from Theorem~\ref{th:pold}, see also~\cite{bondnek}). On the other hand, $\mathcal{P}_1(\alb)$ has non-contracting self-similar subgroups, for example the ``long range'' group generated by
\[a=\sigma(1, a),\qquad b=(a, b).\]

It follows from the second paragraph of Theorem~\ref{th:pold} that the limit space of a contracting subgroup of $\mathcal{P}_d(\alb)$ is one-dimensional.

As an example of a self-similar subgroup of $\mathcal{P}_0(\alb)$, consider the Hanoi Tower group (see~\cite{grisunik:hanoi}) generated by
\begin{align*}
a &=(01)(1, 1, a),\\
b &=(02)(1, b, 1),\\
c &=(12)(1, 1, c)
\end{align*}

Let us apply Theorem~\ref{th:contractingmodeldim} to it. Let $\Gamma$ be the graph of groups equal to the tripod with the vertex groups of the feet $\Z/2\Z$ and all the other vertex and edge groups trivial (see the left-hand side half of Figure~\ref{fig:sierpinski}). Let $\Gamma_1$ be the graph of groups shown on the right-hand side half of Figure~\ref{fig:sierpinski} also with $\Z/2\Z$ at the vertices of valency 1 and trivial everywhere else. We have a natural covering map $f:\Gamma_1\arr\Gamma$ folding $\Gamma_1$ at the vertices of valency 2 and mapping the vertices $A_0, B_1, C_2$ and $A, B, C$ to $A, B, C$, respectively. Let $\iota:\Gamma_1\arr\Gamma$ be the map multiplying the lengths of all edges by $1/2$, mapping the vertices $A, B, C$ of valency 2 to the center of the tripod $\Gamma$, the vertices of degree 3 to the midpoints of the legs of the tripod, and the vertices $A_0, B_1, C_2$ of valency 1 to its feet $A, B, C$, respectively. 

\begin{figure}
\centering
\includegraphics{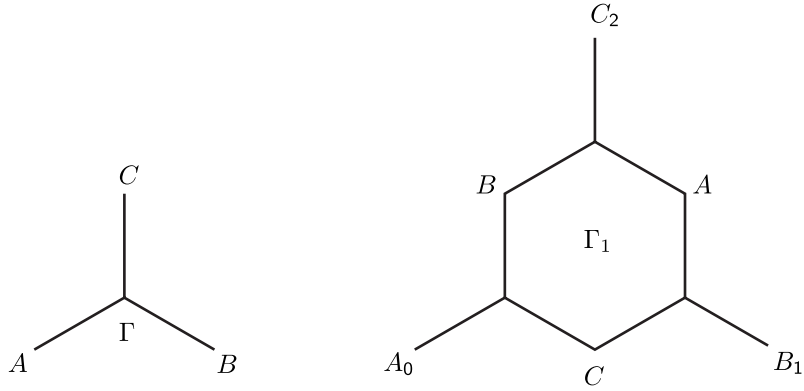} 
\caption{Models of the Hanoi tower group}
\label{fig:sierpinski}
\end{figure}

It is checked directly that the iterated monodromy group of the correspondence $f, \iota:\Gamma_1\arr\Gamma$ is the group $\langle a, b, c\rangle$ defined above. Since $\iota$ is contracting, the limit space of this group has topological dimension 1. In fact, it is the classical Sierpinski gasket (triangle). See an iteration $\Gamma_5$ of the correspondence $f, \iota$ on Figure~\ref{fig:hanoiiteration}.

\begin{figure}
\centering
\includegraphics{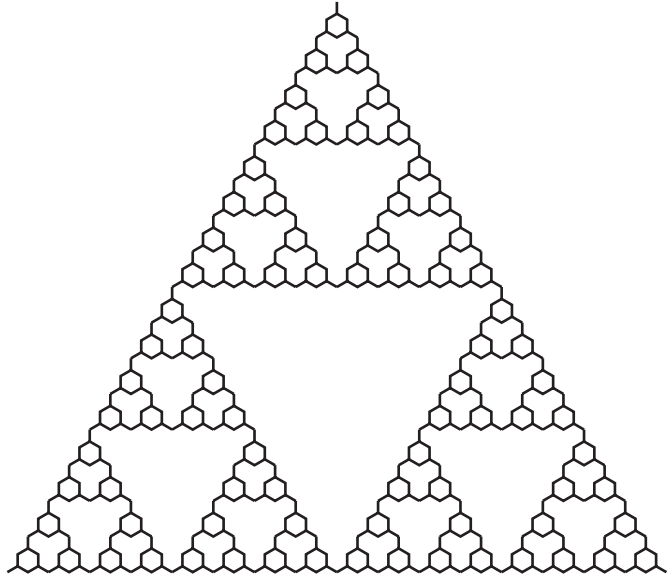}
\caption{Sierpinski gasket}
\label{fig:hanoiiteration}
\end{figure}

Let us apply Theorem~\ref{th:dad}. Figure~\ref{fig:hanoischr} shows the graph of the action on the second level of the tree. Since all generators are involutions, we do not orient the edges.

\begin{figure}
\centering
\includegraphics{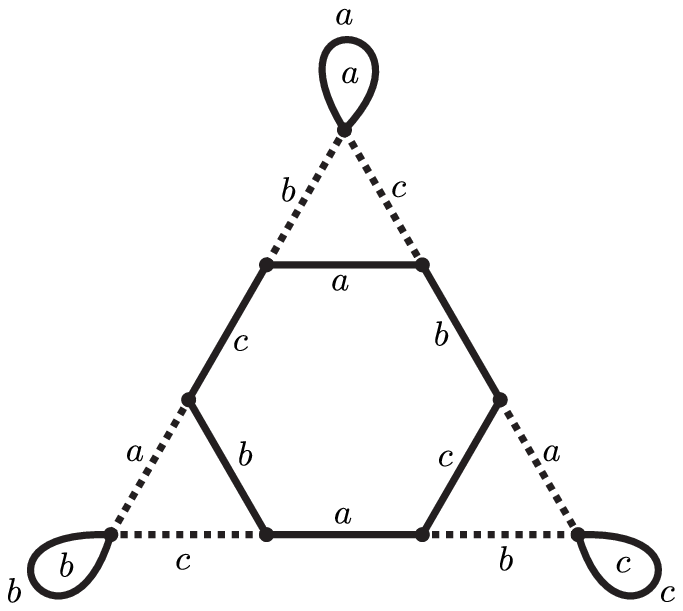}
\caption{The action on the second level}
\label{fig:hanoischr}
\end{figure}

Let us partition the second level $\{0, 1, 2\}^2$ into two parts: $\{00, 11, 22\}$ and the rest. We have marked the edges connecting vertices from different parts by dashed lines on Figure~\ref{fig:hanoischr}. The sections of the generators in the words $\{0, 1, 2,\}^2\setminus\{00, 11, 22\}$ are trivial, so the isomorphism between the corresponding subtrees are of the form $x_1x_2w\mapsto y_1y_2w$, and we get therefore a finite groupoid. The vertices $00$, $11$, and $22$ are not connected by the generators, and since $a, b, c$ are of order two, the conditions of Theorem~\ref{th:dad} are satisfied for this partition, so the limit space is one-dimensional.

As an example of a contracting subgroup of $\mathcal{P}_1(\alb)$, consider the iterated monodromy group of $1-\frac{3+\sqrt{5}}{2z^2}$ generated by
\[a=\sigma(c, b^{-1}),\quad b=(a, 1),\quad c=\sigma(1, c^{-1}).\]

See its limit space $\lims$, which is homeomorhic to the Julia set of $1-\frac{3+\sqrt{5}}{2z^2}$ on Figure~\ref{fig:lineargr}.

\begin{figure}
\centering
\includegraphics{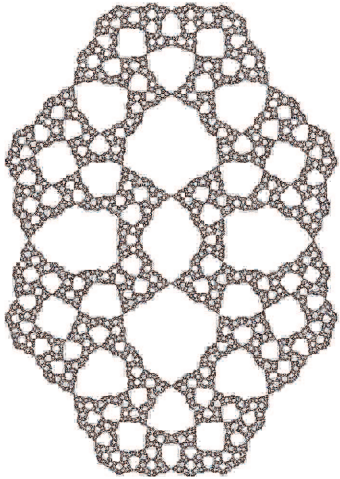}
\caption{The Julia set of $1-\frac{3+\sqrt{5}}{2z^2}$}
\label{fig:lineargr}
\end{figure}



%
%

\subsection{Locally disconnected examples}

Consider the self-similar action of $\Z$ generated by
\[a=(012)(1, 1, a^2).\]
The associated virtual endomorphism is $a^n\mapsto a^{2n/3}$ with the domain equal to $3\Z$. It is contracting, but not self-replicating, since the range of the virtual endomorphism is $2\Z$.

As an example of a contracting model, we can take $\Gamma_0, \Gamma_1=\R/\Z$, the covering $f:x\mapsto 3x$ and the reduction map $\iota:x\mapsto 2x$. We introduce the metric on $\Gamma_1$ with respect to which $f$ is a local isometry, so that the length of the circle $\Gamma_1$ is 3. Then $\iota$ locally multiplies the metric by $2/3$, i.e., is contracting. It is checked directly that the iterated monodromy group of the correspondence $f, \iota:\Gamma_1\arr\Gamma_0$ is given by the above wreath recursion.

The limit space of the group $\langle a\rangle$ is the inverse limit of the circles $\R/\Z$ with respect to the angle doubling maps $\iota_n:x\mapsto 2x$, i.e., the classical binary solenoid.

As another disconnected example, consider the \emph{universal Grigorchuk group} $G$ generated by the following elements acting on $\xs=\{1, 2, 3, 4, 5, 6\}^*$:
\begin{align*}
a &= (12)(34)(56),\\
b &= (a, b, a, b, 1, b),\\
c &= (a, c, 1, c, a, c),\\
d &= (1, d, a, d, a, d).
\end{align*}

It is checked directly that the wreath recursion is contracting for the free product $\langle a\rangle*\langle b, c, d\rangle\cong(\Z/2\Z)*(\Z/2\Z)^2$.

The action of the group on the tree $\xs$ is not level-transitive. For any choice of a sequence $A_n\in\{\{1, 2\}, \{3, 4\}, \{5, 6\}\}$, the binary subtree $A_1\cup A_1\times A_2\cup A_1\times A_2\times A_3\cup\ldots$ is $G$-invariant, and the action of $G$ on it is level-transitive. If we restrict $G$ to this subtree, and take the quotient of $G$ by the kernel of the action, we get one of the \emph{Grigorchuk groups} from~\cite{grigorchuk:growth_en}. For example, if we choose $(A_1, A_2, \ldots)$ to be the periodic sequence $\{1, 2\}, \{3, 4\}, \{5, 6\}, \{1, 2\}, \{3, 4\}, \{5, 6\}, \ldots$, we get the famous \emph{first Grigorchuk group} from~\cite{grigorchuk:80_en}.

A model of the group $G$ consists of the segment of groups $\Gamma_0$   with one vertex group $\Z/2\Z$ identified with $\{1, a\}$ and the other vertex group $(\Z/2\Z)^2$ identified with the subgroup $\{1, b, c, d\}$ of $G$. The graph $\Gamma_1$ consists of three disjoint segments, each of them with two vertex groups $\{1, b, c, d\}$ at the ends. The covering $f:\Gamma_1\arr\Gamma_0$ folds each of the segments in two, maps the folding point to the end with vertex group $\Z/2\Z$, and vertex groups of the endpoints identically to the vertex group $(\Z/2\Z)^2$. The reduction $\iota$ maps each of the segments homeomorphically onto $\Gamma_0$, maps one of the vertex groups identically to $(\Z/2\Z)^2$, and maps the other one surjectively onto $\Z/2\Z$ by each of the three maps (one for each connected component of $\Gamma_1$):
\begin{gather*}
\iota_0:b\mapsto a,\quad c\mapsto a,\quad d\mapsto 1;\\
\iota_1:b\mapsto a,\quad c\mapsto 1,\quad d\mapsto a;\\
\iota_2:b\mapsto 1,\quad c\mapsto a,\quad d\mapsto a.
\end{gather*}
Here the covering $f$ is a local isometry, so the connected components of $\Gamma_1$ have lengths equal to 2, and the reduction map $\iota$ locally divides the distances by 2.

The limit space $\lims$ is homeomorphic to the direct product of a segment with the Cantor set.

\subsection{Sierpi\'nski carpet}

The classical Sierpi\'nski carpet can be realized as the limit space of the  group generated by the automaton whose dual Moore diagram is shown on Figure~\ref{fig:carpet}

\begin{figure}
\centering
\includegraphics{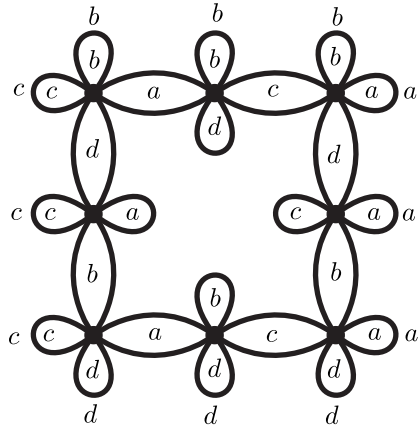}
\caption{Sierpinski carpet automaton}
\label{fig:carpet}
\end{figure}

The corresponding wreath recursion is
\begin{align*}
a&=(12)(67)(1, 1, a, 1, a, 1, 1, a)\\
b&=(46)(58)(b, b, b, 1, 1, 1, 1, 1)\\
c&=(23)(78)(c, 1, 1, c, 1, c, 1, 1)\\
d&=(14)(35)(1, 1, 1, 1, 1, d, d, d)
\end{align*}

It is easy to see that the generators $a, b, c, d$ are of order 2. The limit space of this group is the classical Sierpinski carpet. The limit dynamical system folds it naturally and then stretches it to the original carpet, so that the corner squares are mapped onto the whole carpet by an orientation preserving similarity, while the other four squares are mapped to it by an orientation reversing similarities.

\begin{figure}
\centering
\includegraphics{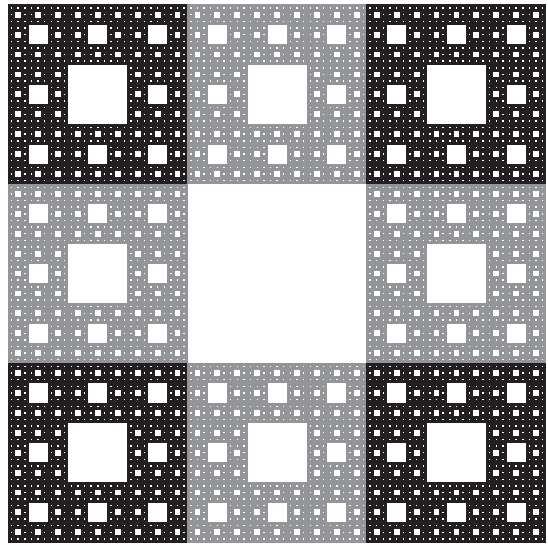}
\caption{Sierpinski carpet}
\label{fig:scarpet}
\end{figure}

The wreath recursion for the element $ab$ is shown as a dual Moore diagram on Figure~\ref{fig:ab}.

\begin{figure}
\centering
\includegraphics{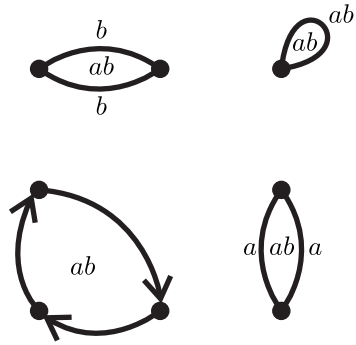}
\caption{The automaton for $ab$}
\label{fig:ab}
\end{figure}

It follows from the recursion that $ab$ is an element of order 6. Consequently, $\langle a, b\rangle$ is the dihedral group $D_6$ of order 12. Same is true for the groups $\langle b, c\rangle$, $\langle c, d\rangle$, and $\langle a, d\rangle$.


Consequently, the partition of the first level into two sets shown on Figure~\ref{fig:partition} satisfies the conditions of Theorem~\ref{th:dad}, which implies that the limit space of this group has topological dimension 1.

\begin{figure}
\centering
\includegraphics{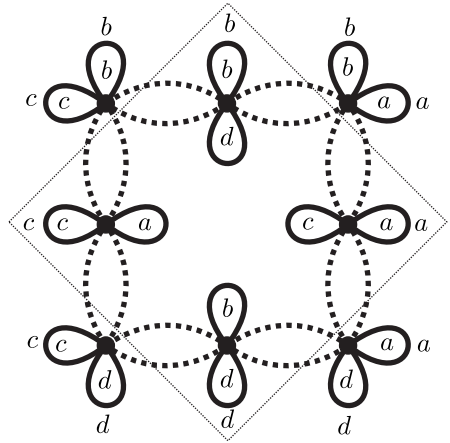}
\caption{A partition of the first level}
\label{fig:partition}
\end{figure}

Let us illustrate Theorem~\ref{th:contractingmodeldim} by this example. We will construct a contracting model of the second iteration of the limit dynamical system. Define $\Gamma_0$ as the graph of groups consisting of a cycle of length four with the vertex groups $\langle a, b\rangle$, $\langle a, d\rangle$, $\langle c, d\rangle$, and $\langle b, c\rangle$, and four edge groups $\langle a\rangle$, $\langle b\rangle$, $\langle c\rangle$, $\langle d\rangle$, with the  identical embeddings. By imitating the second iteration of the limit dynamical system, we get a covering graph $\Gamma_2$  shown on Figure~\ref{fig:model}.

\begin{figure}
\centering
\includegraphics{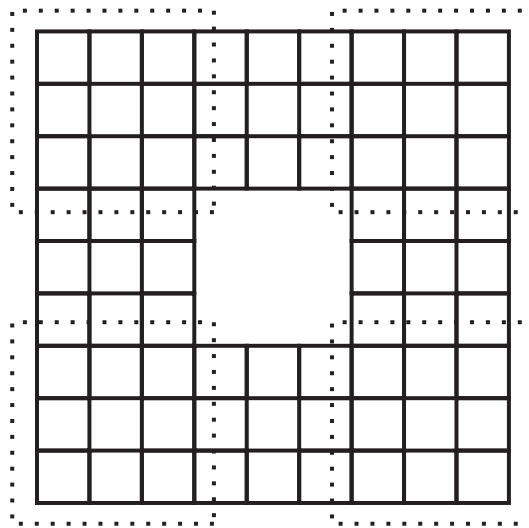}
\caption{Contracting model}
\label{fig:model}
\end{figure}

A contracting map $\iota$ will map the four $3\times 3$ corners to the vertices of $\Gamma_0$, the four $3\times 3$ squares between them to the corresponding edges. We will get a map multiplying the lengths of the edges of $\Gamma_2$ by $1/3$ or $0$. It is checked directly that, with appropriate structure of a graph of groups on $\Gamma_2$, the iterated monodromy group of the described correspondence is $\langle a, b, c, d\rangle$. Theorem~\ref{th:contractingmodeldim} shows then that the limit space has topological dimension 1. The model also implies that the wreath recursion is contracting on the fundamental group of the graph of groups $\Gamma_0$, which is virtually free as the fundamental group of a finite graph of groups with finite vertex groups.

\providecommand{\bysame}{\leavevmode\hbox to3em{\hrulefill}\thinspace}
\providecommand{\MR}{\relax\ifhmode\unskip\space\fi MR }
\providecommand{\MRhref}[2]{%
  \href{http://www.ams.org/mathscinet-getitem?mr=#1}{#2}
}
\providecommand{\href}[2]{#2}

\end{document}